\documentclass[reqno,11pt]{amsart}
\usepackage{amsmath, latexsym, amsfonts, amssymb, amsthm, amscd}
\usepackage{graphics,epsf,psfrag}
\usepackage{graphics,epsf,psfrag}
 \usepackage[utf8]{inputenc}
\usepackage{stmaryrd}
\usepackage{mathabx}
\usepackage{hyperref}
\usepackage{xcolor}

\renewcommand{\bar}{\overline}

\newcommand{\lint}{\llbracket}
\newcommand{\rint}{\rrbracket}

\numberwithin{equation}{section}

\newtheorem{theorem}{Theorem}[section]
\newtheorem{lemma}[theorem]{Lemma}
\newtheorem{proposition}[theorem]{Proposition}

\newtheorem{rem}[theorem]{Remark}

\newcommand{\Diam}{\mathrm{Diam}}
\newcommand{\dd}{\mathrm{d}}
\newcommand{\ind}{\mathbf{1}}

\renewcommand{\tilde}{\widetilde}
\renewcommand{\hat}{\widehat}
\newcommand{\cc}{\complement}

\newcommand{\cA}{{\ensuremath{\mathcal A}} }
\newcommand{\cB}{{\ensuremath{\mathcal B}} }
\newcommand{\cF}{{\ensuremath{\mathcal F}} }

\newcommand{\cD}{{\ensuremath{\mathcal D}} }

\newcommand{\cZ}{{\ensuremath{\mathcal Z}} }

\newcommand{\bP}{{\ensuremath{\mathbf P}} }
\newcommand{\bQ}{{\ensuremath{\mathbf Q}} }
\newcommand{\bE}{{\ensuremath{\mathbf E}} }

\DeclareMathSymbol{\leqslant}{\mathalpha}{AMSa}{"36} 
\DeclareMathSymbol{\geqslant}{\mathalpha}{AMSa}{"3E} 
\DeclareMathSymbol{\eset}{\mathalpha}{AMSb}{"3F}     

\newcommand{\bbE}{{\ensuremath{\mathbb E}} }

\newcommand{\bbN}{{\ensuremath{\mathbb N}} }

\newcommand{\bbP}{{\ensuremath{\mathbb P}} }

\newcommand{\bbR}{{\ensuremath{\mathbb R}} }

\newcommand{\bbZ}{{\ensuremath{\mathbb Z}} }

\newcommand{\gep}{\varepsilon}       

\newcommand{\gO}{\Omega}
\newcommand{\gl}{\lambda}

\makeatletter
\def\captionfont@{\footnotesize}
\def\captionheadfont@{\scshape}

\long\def\@makecaption#1#2{%
  \vspace{2mm}
  \setbox\@tempboxa\vbox{\color@setgroup
    \advance\hsize-6pc\noindent
    \captionfont@\captionheadfont@#1\@xp\@ifnotempty\@xp
        {\@cdr#2\@nil}{.\captionfont@\upshape\enspace#2}%
    \unskip\kern-6pc\par
    \global\setbox\@ne\lastbox\color@endgroup}%
  \ifhbox\@ne 
    \setbox\@ne\hbox{\unhbox\@ne\unskip\unskip\unpenalty\unkern}%
  \fi
  \ifdim\wd\@tempboxa=\z@ 
    \setbox\@ne\hbox to\columnwidth{\hss\kern-6pc\box\@ne\hss}%
  \else 
    \setbox\@ne\vbox{\unvbox\@tempboxa\parskip\z@skip
        \noindent\unhbox\@ne\advance\hsize-6pc\par}%
\fi
  \ifnum\@tempcnta<64 
    \addvspace\abovecaptionskip
    \moveright 3pc\box\@ne
  \else 
    \moveright 3pc\box\@ne
    \nobreak
    \vskip\belowcaptionskip
  \fi
\relax
}
\makeatother
\def\writefig#1 #2 #3 {\rlap{\kern #1 truecm
\raise #2 truecm \hbox{#3}}}

\title{Critical Gaussian Multiplicative Chaos for singular measures}

\author{Hubert Lacoin}
\address{
  IMPA, Institudo de Matem\'atica Pura e Aplicada, Estrada Dona Castorina 110
Rio de Janeiro, CEP-22460-320, Brasil. 
}

\begin{document}

 \begin{abstract}
 Given $d\ge 1$, we provide a construction of the random measure - the  critical Gaussian Multiplicative Chaos - formally defined $e^{\sqrt{2d}X}\dd \mu$
where $X$ is a $\log$-correlated Gaussian field and $\mu$ is a locally finite  measure on $\bbR^d$.
Our construction generalizes the one performed in the case where $\mu$ is the Lebesgue measure. It requires that the measure $\mu$ 
is sufficiently spread out, namely that for $\mu$ almost every $x$ we have 
$$  \int_{B(0,1)}\frac{\mu(\dd y)}{|x-y|^{d}e^{\rho\left(\log \frac{1}{|x-y|} \right)}}<\infty, $$
for any compact set where $\rho:\bbR_+\to \bbR_+$ can be chosen to be any lower envelope function for the $3$-Bessel process (this includes $\rho(x)=x^{\alpha}$ with $\alpha\in (0,1/2)$).
We prove that three distinct random objects converge to a common limit which defines the critical GMC: the derivative martingale, the critical martingale, and the exponential of the mollified field. We also show that the above criterion for the measure $\mu$ is in a sense optimal.
\\[10pt]
  \textit{}
 \end{abstract}

\maketitle

 
 \section{Introduction}

We consider $K: \bbR^d\times \bbR^d \to (-\infty,\infty]$ to be a positive definite kernel on $\bbR^d$ ($d\ge 1$ is fixed) which admits a decomposition in the following form
 \begin{equation}\label{fourme}
  K(x,y):= \log \frac{1}{|x-y|}+L(x,y),
 \end{equation}
where $L$ is a continuous function ($\log 0=\infty$ by convention). 
A kernel $K$ is  positive definite if for any  $\rho\in C_c(\bbR^d)$ ($\rho$ continuous with compact support)
 $$\int_{\cD^2 }  K(x,y) \rho(x)\rho(y)\dd x \dd y\ge 0.$$
Given  $\mu$ a locally finite measure on $\bbR^d$, 
the \textit{Gaussian Multiplicative Chaos} with reference measure $\mu$ and intensity $\alpha>0$, is formally defined as the random measure which is obtained by integrating the exponential of a centered Gaussian field $X$ with covariance $K$, that is to say
 \begin{equation}\label{defGMC}
 e^{\alpha X(x)} \mu(\dd x) 
 \end{equation}
The main difficulty that comes up when  trying to give a mathematical interpretation to the expression  \eqref{defGMC} is 
that a field with a covariance of the type \eqref{fourme} can be defined only as a random distribution and thus $X$ is not defined pointwise.
The problem of rigorously defining GMC was introduced by Kahane in \cite{zbMATH03960673}. The standard procedure (for any $\alpha\ge 0$) is to  use a sequence of  approximation of $X$ 
and then pass to the limit. Mostly  two kinds of approximation of $X$ have been considered in the literature: 
\begin{itemize}
 \item [(A)] A mollification of the field, $X_{\gep}$,  via convolution with a smooth kernel on scale $\gep$,
 \item [(B)] A martingale approximation, $X_t$, via an integral decomposition of the kernel $K$.
\end{itemize}
Rigorous definitions of $X_{\gep}$ and $X_t$ are given in the next subsections.
The idea is then to consider the measure $M^{\alpha}_\gep(\cdot,\mu)$ defined by
\begin{equation}\label{defalpha}
M^{\alpha}_\gep(f,\mu):=\int_{\bbR^d} f(x) e^{\alpha X_{\gep}(x)-\frac{\alpha^{2}}{2}\bbE[X_{\gep}(x)]} \mu(\dd x),
\end{equation}
or $M^{\alpha}_t$, defined analogously for the martingale approximation $X_t$, and let either $\gep\to 0$ or $t\to \infty$.
For context, let us briefly review the main results that have been obtained concerning the convergence of $M^{\alpha}_t$, $M^{\alpha}_\gep$ in the case where $\mu$ is the Lebesgue measure.
It has been proved that 
for $\alpha\in [0,\sqrt{2d})$ (see \cite{Natele,Shamov} for the latest contributions) that both $M^{\alpha}_t$ and $M^{\alpha}_\gep$  converge to a nontrivial limit $M^{\alpha}$, which does not
depend on the 
mollifier (we refer the reader to the introduction of \cite{Natele} or to the review \cite{RVreview} for an account of the steps leading to this result).
When $\alpha=\sqrt{2d}$, the asymptotic behavior of  $M^{\alpha}_\gep$  differs from the case $\alpha\in [0,\sqrt{2d})$. We have $\lim_{\gep\to 0}M^{\sqrt{2d}}_\gep(f)=0$ for any bounded measurable $f$ and a 
renormalization procedure is required to obtain a nontrivial limit. 
More precisely it has been proved that $\sqrt{\pi\log (1/\gep)/2} M^{\sqrt{2d}}_\gep$ converges to a nontrivial limit that does not depend on the mollifier. This was achieved in three steps, starting 
with the martingale approximation $X_t$: 
\begin{itemize}
 \item [(i)] In \cite{MR3262492}, the convergence of the random (signed) measure  
 \begin{equation}\label{derivtiv}
 D_t(\dd x):=(\sqrt{2d}\bbE[X_{t}(x)]- X_t) e^{\sqrt{2d} X_{t}(x)- d \bbE[X_{t}(x)]} \dd x,
 \end{equation}
 to a nontrivial nonnegative measure $D_\infty$ was established.
 \item [(ii)] In \cite{MR3215583}, it was proved that $\sqrt{\pi t/2}  M^{\sqrt{2d}}_t$ converges to the same limit.
 \item [(iii)] The  convergence of $\sqrt{\pi \log (1/\gep)/2} M^{\sqrt{2d}}_\gep$ towards $D_{\infty}$ was shown in  \cite{MR3613704}. 
\end{itemize}
We refer to \cite{MR4396197} for a throughout reviews of these steps and to \cite{critix} for an alternative short and self-contained proof of these three statements.

\medskip

In the case of a general measure $\mu$, the convergence of $M^{\alpha}_{\gep}(f,\mu)$ requires $\mu$ to be sufficiently spread out.
More specifically, it has been proved in \cite{Natele} that $M^{\alpha}_{\gep}(f,\mu)$ converges when there exists $\gamma>\alpha^2/2$ such that for any compact set $K\subset \bbR^d$ we have
\begin{equation}\label{capasouscrit}
 \int_{K^2}\frac{\mu(\dd x)\mu(\dd y)}{|x-y|^{\gamma}} <\infty.
\end{equation}
The above condition states in particular that any set of positive measure has Hausdorf dimension at least equal to $\gamma$
(see for instance \cite[Theorem 4.11]{falcofractal}).

\medskip

In the present manuscript we are interested in the critical case $\alpha=\sqrt{2d}$ and  we want to extend the results proved in \cite{MR3262492, MR3215583, MR3613704}
to measures which are singular with respect to Lebesgue.
For this to be possible, we need some assumption that quantifies how spread out  the measure $\mu$ is.
The specific requirement we have for $\mu$ requires the introduction of a technical notion.

\medskip

Given $\rho: \bbR_+\to \bbR_+$ an increasing function, we say that $\rho$ is a \text{lower-envelope} for the $3$-Bessel process if 
given $(\beta_t)_{t\ge 0}$ is a $3$-Bessel process (the process given by the norm of a three dimensional Brownian Motion) we almost surely have have
\begin{equation}\label{dvorerdos}
\lim_{t\to \infty} \beta_t /\rho(t)=\infty.
\end{equation}
The  Dvoretzky–Erd\"os test (see for instance \cite[Theorem 3.22]{Peresmortersbrownian}), guarantees that $\rho$ is a lower envelope function if and only if 
\begin{equation}\label{besselcond}
 \int^{\infty}_{1} \rho(u) u^{-3/2}\dd u<\infty.
\end{equation}Examples of lower-envelope functions include 
\begin{equation}\label{examples}
\rho^{(1)}(u)=u^{\gamma} \text{ with } \gamma\in(0,1/2) \quad \text{ and } \quad   \rho^{(2)}(u)=u^{1/2}(\log( u+2))^{-\zeta} \text{ with } \zeta>1.
\end{equation}
In the present paper, we assume that there exists $\rho$ a lower-envelope of the $3$-Bessel process such that for any compact set $K$, $\mu$ satisfies  
\begin{equation}\label{clazomp}
 \int_{K\times K}\frac{\mu(\dd x)\mu(\dd y) \ind_{\{|x-y|\le 1\}}}{|x-y|^{d}e^{\rho\left(\log \frac{1}{|x-y|} \right)}}<\infty.
\end{equation}
We prove that under assumption \eqref{clazomp}, $M^{\sqrt{2d}}_t(\cdot,\mu)$, $M^{\sqrt{2d}}_\gep(\cdot,\mu)$ and $D_t(\cdot,\mu)$ (the latter quantity being defined as in \eqref{derivtiv} but integrating with respect to $\mu$), converge (after an appropriate renormalization in the two first cases), to the same limiting measure $D_\infty(\cdot,\mu)$, which defines the critical GMC with reference measure $\mu$.

\begin{rem}
The condition \eqref{clazomp} is stronger than the one mentioned in the abstract, but it turns out that in practice they are equivalent. Indeed any measure satisfying the former criterion can be arbitrarily well approximated by a measure that satisfies \eqref{clazomp}.
We refer to Sections \ref{extense} and \ref{relax} for more discussion on this issue.
\end{rem}

\section{Model and results}

\subsection{The exponential of a mollified, $\log$-correlated field}\label{molly}

Let us now provide rigorous definitions for the mathematical objects discussed in the introduction.

\subsubsection*{Log-correlated fields  defined as distributions}

Since $K$ is infinite on the diagonal, it is not possible to define directly a Gaussian field indexed by $\bbR^d$ with covariance function $K$.
We consider instead a process indexed by test functions. 
We define $\hat K$ as the following quadratic form on  $C_c(\bbR^d)$  
\begin{equation}\label{hatK}
 \hat K(\rho,\rho')=
\int_{\cD^2}  K (x,y)\rho(x)\rho'(y)\dd x \dd y,
\end{equation}
and define $X= \langle X, \rho \rangle_{\rho\in C_c(\bbR^d)}$ as a the centered Gaussian field indexed by $C_c(\bbR^d)$ with covariance kernel given by $\hat K$.

\subsubsection*{Mollification of the field}

The random distribution $X$ can be approximated by a sequence of  functional fields - that is,  fields indexed by $\bbR^d$ - 
by the mean of mollification.
 Consider $\theta$ a nonnegative function in $C_c^{\infty}(\bbR^d)$ (meaning infinitely differentiable with compact support)  whose support is included in  $B(0,1)$ (for the remainder 
 of the paper $B(x,r)$ denotes the closed 
 Euclidean ball of center $x$ and radius $r$) 
 and which satisfies  $\int_{B(0,1)} \theta(x) \dd x=1.$
We define for $\gep\in(0,1)$,
$\theta_{\gep}:=\gep^{-d} \theta( \gep^{-1}\cdot)$, set
and consider $(X_{\gep}(x))_{x\in \bbR^d}$, the mollified version of $X$, that is
\begin{equation}\label{xmolly}
 X_{\gep}(x):= \langle X, \theta_{\gep}(x-\cdot) \rangle
\end{equation}
One can check that the field $X_{\gep}(\cdot)$ has covariance 
\begin{equation}\label{labig}
K_{\gep}(x,y):=\bbE[X_{\gep}(x)X_{\gep}(y)]=
\int_{\bbR^{2d}} \theta_{\gep}(x-z_1) \theta_{\gep}(y-z_2)
K(z_1,z_2)\dd z_1 \dd z_2.
\end{equation}
We set $K_{\gep}(x):=K_{\gep}(x,x)$ and use a similar convention for other covariance functions.
Since $K_{\gep}$ is infinitely differentiable, by Kolmogorov's criterion (see e.g.\  \cite[Theorem 2.9]{legallSto}),  there exists a continuous modification of $X_{\gep}(\cdot)$ 
(in the remainder of the paper, we always consider the continuous modification of a process when it exists) we can make sense of integrals involving $X_\gep$.

\subsubsection*{Taking the exponential}
 We let $\mathcal B_b$ denote the bounded Borel subsets of $\bbR^d$ and $B_b$ denote the bounded Borel functions with bounded support
 \begin{equation}\label{measurable}
 B_b=B_b(\bbR^d)= \left\{ f \  \text{measurable} \ : \sup_{x\in \bbR^d} |f(x)|<\infty , \ \{ x \ : \ |f(x)|\ne 0\} \in \cB_b  \right\},
 \end{equation}
For a fixed locally finite Borel measure $\mu$, we  define a random measure     $M_{\gep}$ by setting for $f\in B_b$ (recall \eqref{defalpha}, we chose for the remainder of the paper not to underline
   the dependence in $\alpha$ and $\mu$ in the notation for better readability)
 \begin{equation}\label{forboundedfunctions}
    M_{\gep}(f):=\int_{\bbR^d} f(x) e^{\sqrt{2d} X_{\gep}(x)- d K_\gep(x)}\mu(\dd x).
 \end{equation}
 We set $M_{\gep}(E):=M_{\gep}(\ind_{E})$ for $E\in \mathcal B_b$ and keep a similar convention for other measures.

\subsection{Star-scale invariance and our assumption on $K$}\label{starscaleconstru}

We assume throughout the paper that the kernel $K$ has an \textit{almost star-scale invariant} part (see Remark \ref{comment} below for a comment concerning this assumption). 
Following a terminology introduced in \cite{junnila2019}, we say that a the kernel $K$ defined on $\bbR^d$ is \textit{almost star-scale invariant}
if it can be written in the form  such that
\begin{equation}\label{iladeuxstar}
 \forall x,y \in \bbR^d,\  K(x,y)=\int^{\infty}_{0} (1-\eta_1 e^{-\eta_2 t}) \kappa(e^{t}(x-y))\dd t,
\end{equation}
 where  $\eta_1\in[0,1]$ and $\eta_2>0$ are constants and the function $\kappa\in C^{\infty}_c(\bbR^d)$ is radial, nonnegative and definite positive. More precisely we assume the following:
\begin{itemize}
 \item [(i)] $\kappa\in C_c^{\infty}(\bbR^d)$ and there exists $\tilde \kappa : \ \bbR^+\to [0,\infty)$ such that $\kappa(x):=\tilde \kappa(|x|)$,
 \item [(ii)] $\tilde \kappa(0)=1$ and $\tilde\kappa(r)=0$ for $r\ge 1$,
 \item [(iii)]The mapping $(x,y)\mapsto \kappa(x-y)$ defines a positive definite-kernel on $\bbR^d\times \bbR^d$ or equivalently the Fourier transform of $\kappa$ satisfies $\hat \kappa(\xi)\ge 0$ for all $\xi\in \bbR´^d$. 
\end{itemize}
We say that a kernel $K$   \textit{has an almost star-scale invariant part}, if 
\begin{equation}\label{3star}
 \forall x,y\in \bbR^{d}, \   K(x,y)=K_0(x,y)+ \bar K(x,y)
\end{equation}
where $\bar K(x,y)$ is an almost star-scale invariant kernel and $K_0$ is positive  definite and H\"older continuous on $\bbR^{2d}$. 
 Given $K$ with an almost star-scale invariant part, and using the decomposition \eqref{iladeuxstar} for $\bar K$, we set
\begin{equation}\label{defqt}
 Q_t(x,y):=  \kappa(e^{t'}(x-y))
 \end{equation}
where $t'$ is defined as the unique positive solution of $ t'-\frac{\eta_1}{\eta_2}(1-e^{-\eta_2 t'})=t.$
We set 
\begin{equation}\label{kkttt}\begin{split}
 K_t(x,y)&:=K_0(x,y)+ \int^t_0 Q_s(x,y) \dd s \\
 &=K_0(x,y) + \int_{0}^{t'} (1-\eta_1 e^{-\eta_2 s}) \kappa(e^{s}(x-y))\dd s=:K_0(x,y)+ \bar K_t(x,y).
\end{split}\end{equation}
Note that we have $\lim_{t\to \infty} K_t(x,y)=K(x,y)$ and $\bar K_t(x)=t$.
If $K$ satisfies \eqref{iladeuxstar} then
\begin{equation}
 L(x,y):= K(x,y)+\log |x-y|,
\end{equation}
can be extended to a continuous function on $\bbR^{2d}$, so that a kernel $K$ with an almost star-scale invariant part
can always be written in the form \eqref{fourme}.

\begin{rem}
 There is an obvious conflict of notation between $K_t$ introduced above  and $K_{\gep}$ introduced in \eqref{labig} 
 and the same can be said about $X_t$ and $M_t$ introduced in the next section. However this abuse should not cause any confusion since we will keep using the letter 
 $\gep$ for quantities related to the mollified field $X_{\gep}$ and latin letters for quantities related to the martingale approximation $X_t$.  
\end{rem}

\subsection{Convergence of the mollified critical GMC}

Our first main result is the convergence of the measure $M_{\gep}$ - properly rescaled - 
towards a limit $M'$. 
We need to specify a topology on the set of locally finite measure. We say that a sequence of locally finite measures 
$(\mu_n)$ converges weakly to $\mu$ if
\begin{equation}\label{wikiwik}
\forall f \in C_c(\bbR^d), \quad  \lim_{n\to \infty} \int_{\bbR^d} f(x)\mu_n(\dd x)= \int_{\bbR^d} f(x) \dd \mu(\dd x).
\end{equation}
The topology of weak convergence for locally finite nonnegative measures is metrizable and separable, hence we can associate to it a notion of convergence in probability for a sequence of 
 random measure. Recall that the topological support of a measure is defined as the complement of the largest open set with zero measure (or as the smallest closed set with full measure).

\begin{theorem}\label{mainres}
 If $X$ is a centered Gaussian field whose covariance kernel $K$  has an almost star-scale invariant part, $\mu$ is a locally finite measure that satisfies \eqref{clazomp} and  $E\in \mathcal B_b$ is such that $\mu(E)>0$. Then
there exists an a.s.\ positive random variable $M'(E)$ such that the following convergence holds in probability
 \begin{equation}\label{inproba}
  \lim_{\gep \to 0}\sqrt{\frac{\pi\log (1/\gep)}{2}}M_{\gep}(E)=M'(E).
 \end{equation}
 The limit satisfies $\bbE[M'(E)]=\infty$ and  does not depend on the specific mollifier $\theta$ used to define $X_{\gep}$. Furthermore 
 there exists a modification of the process $(M'(E))_{E\in \mathcal B_b}$ which is a locally finite random measure $M'$ such that the following convergence holds weakly in probability
 \begin{equation*}
    \lim_{\gep \to 0}\sqrt{\frac{\pi\log (1/\gep)}{2}}M_{\gep}=M'.
 \end{equation*}
  $M'$ is atomless and its  topological support a.s.\ coincides with that of $\mu$ 
\end{theorem}
\begin{rem}
When $\mu(E)=E$ it is immediate that $M_{\gep}(E)=0$ for every $\gep\in(0,1)$ and hence \eqref{inproba} also holds in that case with $M'(E)=0$.
\end{rem}

\begin{rem}\label{comment}
The assumption \eqref{iladeuxstar} may seem at first restrictive, but it has been shown in 
\cite{junnila2019} that it is \textit{locally} satisfied as soon as  $L$ is sufficiently regular. The fact allows to extend the result to all 
sufficiently regular kernels defined on an arbitrary domain in $\bbR^d$. This covers all practial applications (including all variants of the  two dimensional Gaussian Free Field) We refer to \cite[Appendix C]{critix} for more detailed statements and proofs (both of which remain valid in our context).
\end{rem}

\subsection{The martingale decomposition of $X$} \label{martindecoco}

To prove Theorem \ref{mainres}, we rely on a standard tool introduced in  \cite{zbMATH03960673}:
A martingale decomposition of the field $X$ which based on the integral decomposition $K$ provided the star-scale assumption.
Given $K$ a kernel on $\bbR^d$ with an almost star-scale invariant part, 
we define $(X_{t}(x))_{x\in \bbR^d, t\ge 0}$ to be a centered Gaussian field with covariance given by 
(using the shorthand notation $a\wedge b:=\min(a,b)$)
\begin{equation}\label{kstxy}
 \bbE[X_t(x)X_s(y)]= K_{s\wedge t}(x,y).
\end{equation}
Since  $(s,t,x,y)\mapsto K_{s\wedge t}(x,y)$ is H\"older continuous, 
the field admits a modification which is continuous in both $t$ and $x$.
We let  
$\mathcal F_t:= \sigma\left( (X_s(x))_{ x\in \bbR^d ,s\in [0,t]}\right)$
denote the natural filtration associated with $X_{\cdot}(\cdot)$.
Note that the process $X$ indexed by $C_c(\bbR^d)$ and defined by
`
\begin{equation}\label{defofX}
\langle X,\rho\rangle= \lim_{t\to \infty} \int_{\bbR^d} X_t(x) \rho(x)\dd x,
\end{equation}
is a centered Gaussian field, so that $X_t$ indeed is an approximation sequence for a $\log$-correlated field with covariance $K$.
We define the measure     $M_{t}$ by setting for $f\in B_b$ (recall \eqref{measurable})  
 \begin{equation}
    M_{t}(f):=\int_{\bbR^d} f(x)  e^{\sqrt{2d} X_{t}(x)- d K_t(x)}\mu( \dd x),
 \end{equation}
By independence of the increments of $X$ (see Section \ref{considerations}), $(M_t(f))$ is an $(\cF_t)$-martingale. 
We also consider the  \textit{derivative martingale} $D_t$, defined by 
\begin{equation}
 D_t(f)= \int_{\bbR^d}f(x)\left( \sqrt{2d}K_t(x)-X_t \right)  e^{\sqrt{2d} X_{t}(x)- d K_t(x)} \mu(\dd x).
\end{equation}
If we define $M^{\alpha}_t$ for $\alpha\ge 0$ like in \eqref{defalpha}, then we have $D_t(f):=- \partial_{\alpha}     M^{\alpha}_{t}(f)|_{\alpha=\sqrt{2d}}$ (hence the name derivative martingale). 
The first step to prove Theorem \ref{mainres} is to establish the almost-sure convergence of $D_t$. 
We define 
\begin{equation}
 \bar D_\infty(f):=\limsup_{t\to \infty} D_{t}(f).
\end{equation}
 \begin{theorem}\label{derivmartin}
Under the same assumptions as Theorem \ref{mainres}, for any fixed $E\in \mathcal B_b$ satisfying $\mu(E)>0$, we have almost-surely
 \begin{equation}
\bar  D_\infty(E)= 
 \lim_{t\to \infty} D_{t}(E)\in(0,\infty).
 \end{equation}
 and $\bbE[\bar D_{\infty}(E)]=\infty$. Furthermore
 there exists a measure $D_{\infty}$ such that almost surely, $D_t$ converges weakly to $D_{\infty}$. For every $f\in  B_b$, $\bbP[ \bar D_\infty(f)=   D_\infty(f)]=1.$
The measure $D_{\infty}$ is atomless and has the same topological support as $\mu$.
 \end{theorem}
\noindent We also prove that $M_t$, when appropriately rescaled, converges to the same limit.

 \begin{theorem}\label{critmartin}
For any fixed $E\in \mathcal B_b$ we have the following convergence in probability, 
 \begin{equation}
  \lim_{t\to \infty} \sqrt{ \frac{\pi t}{2}}M_{t}(E)=
\bar D_{\infty}(E).
 \end{equation}
 Furthermore  $\sqrt{ \frac{\pi t}{2}}M_{t}$ converges weakly in probability to $D_{\infty}$.
\end{theorem}

\begin{rem} 
 A final observation completing the above results, is that, if $(X_{\gep}(\cdot))_{\gep\in (0,1)}$ is constructed jointly with $(X_t(\cdot))_{t\ge 0}$ on a common probability space (using \eqref{defofX} and \eqref{xmolly}), then the limiting measure $M'$ from Theorem \ref{mainres} coincides with $D_{\infty}$. 
In the remainder of the paper we always consider that  $X_{\gep}$ lives in a probability space that contains a martingale approximation $X_t$. Since the question of the convergence of $M_{\gep}(\cdot)$ solely depends on the law of $(X_{\gep}(x))_{\gep\in(0,1),x\in \bbR^d}$, 
this entails no loss of generality.
 \end{rem}

%
%
%
%

\subsection{Degeneracy of the critical GMC for some singular measure}\label{degens}

To complement Theorem \ref{mainres}, we provide a result in the opposite direction to illustrate the fact 
that the assumption we have made concerning $\mu$ is \textit{optimal} in the sense that we can find measures $\mu$ that barely fail to satisfy \eqref{clazomp} and for which the critical GMC is degenerate. Our example take the form of measure supported by a Cantor set. We restrict to the case $d=1$ for simplicity but the example can be generalized to higher dimension in a straightforward manner.

\subsubsection*{Construction of $\mu$} We let $\vartheta: \bbR_+\to \bbR_+$ be a concave increasing function
such that 
\begin{equation}
 \int^{\infty}_1 u^{-3/2}\vartheta(u)\dd u=\infty.
\end{equation}
We make the extra assumption (mostly for simplicity) that $\rho$ can be written in the form $\vartheta(u)= u^{1/2}L(u)$ where $L$ is a slowly varying function.
Examples include $u^{1/2}(\log u)^{\zeta}$ with $\zeta>-1$ or $u^{1/2}(\log u)^{-1}(\log \log u)^{\zeta}$ with $\zeta>-1$ (with appropriated modification to define the function for small values of $u$ and make it concave).

\medskip

We set $a_n:=1-e^{\vartheta(n)-\vartheta(n-1)}$.
We consider $\mu$ to be the uniform probability on the Cantor set $\gO$ obtained by deleting the middle $a_1$-fraction of the segment $[0,1]$ and  iterating the process (using $a_n$ instead of $a_1$ for the $n$-th step) on all the segments obtained in the preceeding steps. 

\medskip

More precisely, given $n\ge 1$ and $(\alpha_i)_{i=1}^n\in (0,1/2)^n$, we consider $A(\alpha_1,\dots,\alpha_n)$ to be the closed subset of $[0,1]$ 
reccursively defined in the following manner.
\begin{equation*}\begin{split}
 A(\alpha)&:=\left[0,\tfrac{1-\alpha}{2}\right]\cup\left[\tfrac{1+\alpha}{2},1\right],\\
 A(\alpha_1,\dots,\alpha_n)&:=\tfrac{1-\alpha_1}{2}A(\alpha_2,\dots,\alpha_n)\cup\left( \tfrac{1-\alpha_1}{2}A(\alpha_2,\dots,\alpha_n)+ \tfrac{1+\alpha_1}{2}\right),
\end{split}\end{equation*}
where $a A+b:=\{ ax+b \ : \ x\in A\}$. 
We consider the set 
$$\gO:=\bigcap_{n\ge 1} A(a_1,\dots,a_n)$$
and $\mu$ to be the uniform probability on $\gO$ (obtained by taking the limit of the uniform probability on  $A(a_1,\dots,a_n)$).
The following indicates that by taking $\vartheta$ in such a way that   $\int^{\infty}_1 u^{-3/2}\vartheta(u)\dd u$ diverges slowly we can make $\mu$ arbitrarily   close to  satisfying
\eqref{clazomp}.
\begin{lemma}\label{topzzzz}
 For the above construction we have, for any $\alpha>(\log 2)^{-1/2}$
 \begin{equation}
  \int_{\gO\times \gO} \frac{\mu(\dd x) \mu(\dd y)}{ e^{-\alpha \vartheta(\log |x-y|) }}<\infty.
 \end{equation}
\end{lemma}
\noindent The following result,  coupled with Lemma \ref{topzzzz} illustrates that our criterion \eqref{clazomp}, is close to optimal. 
\begin{proposition}\label{degeneracy}
For the measure $\mu$ defined above, we have
\begin{equation}
 \lim_{t\to \infty} D_t(\gO)=\lim_{t\to \infty}\sqrt{t} M_t(\gO)= \lim_{\gep \to 0}\sqrt{\log(1/\gep)} M_{\gep}(\gO)=0.
\end{equation}

\end{proposition}

\subsection{Extension and applications}\label{extense}

\subsubsection*{Relaxing the assumption on $\mu$}
We want to justify here that our results extends to the case where $\mu$ satisfies the weaker condition given in the abstract.
This weaker condition is the following, there exists an increasing function $\rho:\bbR_+\to \bbR_+$  satisfying \eqref{besselcond}  such that
\begin{equation}\label{root}
\mu\left( \left\{ x \ : \  \int_{B(x,1)} \frac{\mu(\dd y)}{|x-y|^{d} e^{-\rho(\log \frac{1}{|x-y|})}}<\infty\right\}\right)=0.
\end{equation}
The reason why the extension is possible is that if $\mu$ satisfies \eqref{root} then we can find a sequence of mesures $\mu_n$ which satisfy \eqref{clazomp} and converges to $\mu$ is sufficiently strong sense.
We simply set 
\begin{equation}\label{defmun}\begin{split}
 A_n&:=\left\{x\in \bbR^d \ : \ \int_{B(x,1)} \frac{\mu(\dd y)}{|x-y|^{d} e^{-\rho(\log \frac{1}{|x-y|})}}\le n  \right\},\\
\mu_n(B)&:= \mu(A_n\cap B).
 \end{split}
\end{equation}
It follows from the definition and \eqref{root} that $\mu_n$ satisfy \eqref{clazomp} and that  $\mu_n(B)\uparrow \mu(B)$ for every Borel set $B$.
The critical GMC with respect to $\mu$ is then obtained by taking the limit of 
$M'(\cdot ,\mu_n)$ when $n\to \infty$ (we provide  justifications for this claim in Section \ref{relax}).

\subsubsection*{Slightly beyond $\log$-correlated field}

Let us present here a variant of the star-scale invariant construction, which yields a kernel which does not quite satisfy \eqref{fourme}.
Given $\kappa$ as in \eqref{starscaleconstru} we set 
$$J(x,y):=\int^{\infty}_0  (1+\alpha(t))\kappa(e^t(x-y)) \dd t  $$
where $\alpha:\bbR_+\to \bbR_+$ is a decreasing function that tends to zero.
We set $\rho(t)=\int^t_0\alpha(u)\dd u$.
It is a simple computation to check that 
\begin{equation}
 J(x,y):= \log \left(\frac{1}{|x-y|}\right)+ \rho\left( \log \frac{1}{|x-y|}\right)+ L(x,y).
\end{equation}
where $L(x,y)$ is a continuous function on $\bbR^{2d}$.
We define $t'(t)$ by the equation 
\begin{equation}
 \int^{t'}_0  (1+\alpha(u)) \dd u=t. 
\end{equation}
We have $t'\le t$. We set (in analogy with $Q_t$ and $K_t$ respectively) 
\begin{equation}\begin{split}
R_t(x,y)&:=  (1+\alpha(t'))\kappa(e^{t'}(x-y)),\\
J_t(x,y)&:=\int^t_0 Q_u(x,y)\dd u=  \int^{t'}_0  (1+\alpha(u))\kappa(e^{u}(x-y)) \dd u.
\end{split}\end{equation}
and consider a the Gaussian  field $Y_t$ with covariance given by 
$$\bbE[Y_t(x)Y_s(y)]=J_{s\wedge t}(x,y),$$
and let $Y_{\gep}$ denote the continuum field defined by
$$ Y_{\gep}(x):=\lim_{t\to \infty} \int_{\bbR^d} \theta_{\gep}(y-x)Y_{\gep}(y)\dd y,$$
where the convergence holds in $L^2$. 
In this setup we can define the corresponding critical martingale and derivative martingale and mollified GMC by setting for $f\in B_b$
\begin{equation}\begin{split}
    N_{\gep}(f)&:=\int_{\bbR^d} f(x)  e^{\sqrt{2d} Y_{t}(x)- d t}\dd x,\\
    N_{t}(f)&:=\int_{\bbR^d} f(x)  e^{\sqrt{2d} Y_{t}(x)- d t}\dd x,\\
    F_{t}(f)&:=\int_{\bbR^d} f(x) (\sqrt{2d}t-Y_t(x)) e^{\sqrt{2d} Y_{t}(x)- d t}\dd x.
 \end{split}\end{equation}

\medskip

\noindent Adapting the argument developed in the proof of Theorems 
\ref{mainres}-\ref{derivmartin}-\ref{critmartin} and that of the proof of Proposition \ref{degeneracy}, we can show the following
\begin{theorem}\label{outro}
 If If $\int^{\infty}_1 \rho(u) u^{-3/2}  \dd u<\infty$ then  $\sqrt{ \frac{\pi \log (1/\gep)}{2}}N_\gep$,
 $\sqrt{ \frac{\pi t}{2}}N_t$ and $F_t$ all converge weakly in probability to the same  limiting measure $N'$ which is atomless, with dense support and satisfies for every set of positive Borel measure
 $$ \bbE\left[ F_\infty(E)\right]=\infty.$$
 
 \medskip
 
 \noindent If $\int^{\infty}_1 \rho(u) u^{-3/2}  \dd u=\infty$ then for every $E\in \mathcal B_b$

  $$\lim_{t\to \infty} F_t(E)=\lim_{t\to \infty}\sqrt{t} N_t(E)=\lim_{\gep \to 0}\sqrt{ \frac{\pi \log (1/\gep)}{2}}N_\gep(E)=0.$$

\end{theorem}

\subsubsection*{Critical Liouville Brownian Motion}

The Liouville Brownian Motion, introduced in \cite{lbm}, is formally defined as the isotropic the diffusion process in $\bbR^2$ equiped with the random metric $e^{\alpha X(z)}d z^2$ where $\alpha$ is a real parameter and $(X(z))_{z\in \bbR^2}$ is a two dimensional Gaussian free-field
(that is a specific $\log$-correlated field whose covariance is given by the two dimensional Green kernel).

\medskip

The process  is built as a time change of the standard two dimensional Brownian motion, and an important step of the construction is to be able to define the GMC $e^{\alpha X(z)}\mu_T(\dd z)$ where $\mu_T$ is the occupation measure of a Brownian Motion (defined by $\mu_T(A)=\int^T_{0} \ind_{\{B_t\in A\}} \dd t $).
When $\alpha<2$, the measure $\mu_T$ satisfies the criterion \eqref{capasouscrit} and the GMC can be defined in the ``standard'' way.

\medskip

In \cite{lioucrit}, the extension of the construction to the case $\alpha=2$ was introduced. In that case the time change $e^{ 2 X(z)}\mu_T(\dd z)$ is a critical GMC. However, the justification of the existence of the GMC in \cite{lioucrit} contains an error. 
The authors prove (see \cite[{Equation (3.20)}]{lioucrit}) that the truncated derivative martingale (the process which we denote as $D^{(q)}_t$ in the present paper see \eqref{trunk1}) is bounded in $L^2$ if the reference measure $\mu$ satisfies
\begin{equation}\label{condx}
\int_{K\times K} \frac{\mu(\dd x)\mu(\dd y)\ind_{\{|x-y|\le 1\}}}{|x-y|^2 (\log\frac 1 {|x-y|}+1)^{3/2}}<\infty 
\end{equation}
 for every compact $K$.
The issue is then that neither $\mu_T$ nor any  of its any restriction of it to a set of positive measure satisfy \eqref{condx}.
The construction presented in the present paper allows to resolve this issue, since for any $p>2$ (see for instance \cite[Lemma 3.2]{lioucrit})
\begin{equation}\label{condx2}
\int_{\bbR^2 \times \bbR^2} \frac{\mu_T(\dd x)\mu_T(\dd y)\ind_{\{|x-y|\le 1\}}}{|x-y|^2 (\log\frac 1 {|x-y|}+1)^{p}}<\infty .
\end{equation}
This obviously implies \eqref{clazomp} for $\rho(u)=u^{1/3}$.

\subsection{Organization of the paper}

\subsubsection*{Outline of the proof}
Let us first underline that the proof of Theorems \ref{mainres}, \ref{derivmartin} and \ref{critmartin} displays a lot of similarity with that presented  in \cite{critix} for the case where $\mu$ is a Lebesgue measure. Although the proof we present in this paper is mostly self contained, since it adds an extra layer of technicality on top of the previously used strategy, the reader may find it useful to first take a look at \cite{critix} to get familiar with the original argument.

\medskip

The starting idea for the proof is to modify the process  $D_t$ by introducing a restriction on the trajectories for  $(\bar X_s(x))_{s\in [0,t]}$ (defined by $\bar X_t(x)=X_t(x)-X_0(x)$)
The important technical novelty brought in this paper is a refinement of the truncation procedure. Instead of asking simply for $\bar X_s(x)$ to stay below $\sqrt{2d}s+q$ (where $q$ is a large constant), we restrict it further and discard the contributions of trajectories that goes above the threshold $\sqrt{2d}s+q- \rho_r(s)$ where $\rho_r$ is a slightly modified version of the function $\rho$ (indexed by an extra parameter $r$).

\medskip

In order to explain how we take advantage of this double truncation procedure in the specific case of $D_t$ (the other cases are similar),
We let $D^{(q)}_t$ denotes the output of the first truncation procedure  and $D^{(q,r)}_t$ that of the second one (see the definitions \eqref{trunk1}-\eqref{trunk2}).
A first step - an idea first implemented in  \cite{MR3262492} - is to observe that with a probability going to $1$ when $q\to \infty$ $D_t(E)$ almost coincides with $D^{(q)}_t(E)$ for every $t$ (more specifically we have $D^{(q)}_t(E)=D_t(E)+qM_t(E)$ on a large probability event  but the second term $qM_t$ does not matter because it can easily be shown to converge to zero). This is because $\sup_{x\in E}\bar X_t(x)$ most likely stays below the threshold we have fixed for every $t\ge 0$  (see Proposition \ref{listov} item (iv) below). This step allows to deduce the convergence of $D_t$ from that of $D^{(q)}_t$  which is a non-negative martingale. What is important then is to check that 
$D^{(q)}_t$ is uniformly integrable (for every $q$) so that the obtained limit  is nonnegenerate.

\medskip

The aim of the second truncation is to help the proof of uniform integrability. We observe that, with our  parametrization, $D^{(q,r)}_t(E)$ and $D^{(q)}_t$ are uniformly close (in $t$) in the $L^1$ sense when $r\to \infty$.
Hence to  prove  the uniform integrability of $D^{(q)}_t(E)$, it is sufficient to show that $D^{(q,r)}_t(E)$ is bounded in $L^2$ for every $r$ (cf. Proposition \ref{metaprop}). Our definition of $D^{(q,r)}_t$ is taken so that it makes the second moment finite while keeping $D^{(q,r)}_t$ close to $D^{(q)}_t$ in the $L^1$ sense. By optimizing this proceedure we obtain the  requirement \eqref{clazomp} for the measure $\mu$.

\subsubsection*{Summary of the content of each section}

The paper is organized as follows:
\begin{itemize}
 \item In Section \ref{prelimlim}, we introduce many technical observations and estimates that we use in our proof. The list of ingredients is slightly larger than the one used in \cite{critix} (where the Lebesgue measure treated). In particular we require a few technical estimates concerning Bessel processes and a criterion to prove uniform integrability via $L^1$ approximation truncation (Proposition \ref{metaprop}). The proofs of the results of this section are presented in Appendix \ref{teklesproofs}.
 \item In Section \ref{caseofdt} we prove the convergence of $D_t$, as well as the property of the limit $D_{\infty}$, that is Theorem \ref{derivmartin},
 \item In Section \ref{caseofmt} we prove the convergence of $M_t$, that is Theorem \ref{critmartin},
 \item In Section \ref{caseofmgep} we prove the convergence of $M_{\gep}$, that is Theorem \ref{mainres}.
 \item In Section \ref{lastwalkz} we prove the degeneracy of the critical GMC on the Cantor set introduced in  Section \ref{degens} that is Proposition \ref{degeneracy}.
 \item In Section \ref{sketch} we sketch the proof of the results discussed in Section \ref{extense}.
\end{itemize}

\medskip

\noindent {\bf Acknowledgements:} The authors is grateful to the authors of \cite{lioucrit} for enlightening discussions, and in particular for letting him know about the existing 
the gap in the proof of the construction of Critical Liouville Brownian Motion. He acknowledges the support of a productivity grant from CNPq and that of a CNE grant from FAPERJ.

 \section{Technical preliminaries}\label{prelimlim}

 For the remainder of the paper (exclusing Section \ref{sketch}), we are going to fix the function $\rho$ which appears in \eqref{clazomp}.
We assume that $\rho$ is concave and that $\rho(0)=0$. This entails no loss of generality since if $\rho$ satisfies \eqref{besselcond}, then there exists a concave function $\bar\rho\ge \rho$ that also satisfies \eqref{besselcond} ($\bar \rho$ can be obtained by taking the convex envelope of the region below the graph of $\rho$). We also assume that $\rho(u)\ge u^{1/4}$ for $u$ sufficiently large. 
 Given $r\ge 0$ we define $\rho_r:\bbR_+\to \bbR_+$
  \begin{equation}\label{rhor}
 \rho_r(u)=\rho(u+r)-\rho(u).
 \end{equation}
 The concavity assumptions on $\rho$ and \eqref{besselcond} imply that  $\lim_{r\to \infty} \rho_r(u)=0$ and the convergence holds uniformly on compact sets.
 Note that if $\rho$ satisfies \eqref{clazomp}, then it is also the case for $\rho_r$ for any $r\ge 0$.
 
 \subsection{Gaussian and Brownian estimates}\label{gbesti}

Let us recall first a couple of basic results concerning Gaussian processes and Brownian Motion.
Firstly, the Cameron-Martin formula, that indicates how the distribution of a Gaussian field is affected by an exponential tilt.

\begin{proposition}\label{cameronmartinpro}

 Let $(Y(z))_{z\in \cZ}$ be a centered Gaussian field indexed by a set $\cZ$. We let  $H$ denote its covariance and $\bP$ denote its law. 
Given $z_0\in \cZ$ let us define $\tilde \bP_{z_0}$ the probability obtained from $\bP$ after a tilt by $Y(z_0)$ that is
\begin{equation}
 \frac{\dd \tilde \bP_{z_0}}{\dd \bP}:= e^{Y(z_0)- \frac{1}{2} H(z_0,z_0)}
\end{equation}
Under $\tilde \bP_{z_0}$, $Y$ is a Gaussian field with covariance $H$,
and mean  $\tilde \bE_{z_0}[ Y(z)]=H(z,z_0).$
\end{proposition}

We let $(B_t)_{t\ge 0}$ and $(\beta_t)_{t\ge 0}$ denote respectively a standard one dimensional Brownian motion and a $3$-Bessel process, starting from $a\ge 0$ and denote their respective law by $\bP_a$ and $\bQ_a$. We also let $\bQ_{a,t}$ denote the
distribution of a Brownian Motion conditioned to remain positive until time $t$, that is 
\begin{equation}
\bQ_{a,t}:= \bP_a\left[ \ \cdot \ \ |  \ \forall s\in [0,t], B_s\ge 0 \right].
\end{equation}
The probability $\bQ_a$ and $\bP_a$ are related by the Doob-McKean identity (see for instance the introduction of \cite{doobmckean}). For any finite $t> 0$ and positive measurable function $F$, we have
 \begin{equation}\label{dmk}
 \bP_a\left[ \frac{B_t}{a}F( (B_s)_{s\in [0,t]})\ind_{\{\forall s\in [0,t], B_s>0\}}\right]= \bQ_a\left[ F( (B_s)_{s\in [0,t]})\right]
\end{equation}
Finally we let $\bP_{a}[ \ \cdot \  | \ B_t=b]$ denote the distribution of the Brownian bridge of length $t$ starting from $a$ and ending at $b$.
We provide below a couple of useful identities and inequalities concerning these processes (proofs are provided in Appendix \ref{laabesse}).
 
\begin{lemma}\label{labesse}
The following holds 
\begin{itemize}
 \item [(i)] Setting $ \mathfrak g_{t}(a):= \int^{a}_0 e^{-\frac{z^2}{2t}} \dd z,$ we have 
 \begin{equation}
  \bP_a\left[ \sup_{s\in[0,t]} B_s \ge 0 \right]=\sqrt{ \frac{2}{\pi t}}\mathfrak g_{t}(a)\le \sqrt{\frac{2}{\pi t}}a.
 \end{equation}
\item[(ii)] If $ab\ge 0$ then
\begin{equation}\label{bridjoux}
 \bP_{a}[ \forall s\in[0,t], \ B_s\ge 0 \  | \ B_t=b]= \left(1-e^{-\frac{2ab}{t}}\right) \le  1\wedge \frac{2ab}{t}. 
\end{equation}
\item[(iii)] If $a, b> 0$ then,
\begin{equation}
 \bP[\forall t>0, \ B_t\le at+b]=1-e^{-2ab}\le 2ab.
\end{equation}
\item[(iv)]  For all $a\ge 0$
\begin{equation}\label{TCDbess}
 \lim_{r\to \infty}\bQ_a\left[ \forall s\ge 0, \  \beta_s\ge \rho_r(s)\right]=1.
\end{equation}
\item[(v)]
For every $a>0$ there exists $t_0(a)$ such that for  $t\ge t_0(a)$ and every Borel set $A\subset C([0,t],\bbR)$
\begin{equation}
  \bQ_{a,t}\left[  (B_s)_{s\in[0,t]}\in A \right]\le (1+\sqrt{2}) \sqrt{\bQ_{a}\left[  (\beta_s)_{s\in [0,t]}\in A \right]}.
\end{equation}
\item[(vi)]
For any $r\ge 0$, we have
\begin{equation}\label{browntobess}
\lim_{t \to \infty}\bQ_{a,t}\left[ \forall s\in [0,t], B_s\ge \rho_r(s) \right]= \bQ_a\left[ \forall t\ge 0, \beta_t\ge \rho_r(t) \right].
\end{equation}
\end{itemize}

\end{lemma}

 \subsection{Simple observations concerning our Gaussian fields}\label{considerations}
By construction the increments of $X_t$ are orthogonal in $L^2$ and hence independent. 
Setting
\begin{equation}\label{barxx}
 X^{(s)}_t(x)= X_{t+s}(x)-X_s(x) \text{ and } \bar X_t(x)= X^{(0)}_t(x),
\end{equation}
the field $(\bar X_{t}(x))_{x\in \bbR^d, t\ge 0}$ has covariance $\bar K_{s\wedge t}(x,y)$ (recall \eqref{kkttt}). 
In particular, since $\bar K_t(x,x)=t$, this implies that for every fixed $x\in \bbR^d$ and $s\ge 0$,
$t\mapsto  \bar X^{(s)}_t(x)$ is a standard Brownian Motion independent of $\mathcal F_s$. Furthermore, recalling \eqref{defqt}, we necessarily have   $t' \in [t,t+\eta_1/\eta_2]$. Our assumption that $\kappa$ is supported on $B(0,1)$ implies that if  $|x-y|\ge e^{-s}$ then  $Q_u(x,y)=0$ for every $u\ge s$ and thus 
  $X^{(s)}_{\cdot}(x)$ and  $X^{(s)}_{\cdot}(y)$ are independent Brownian Motions.
We introduce a third field  $X_{t,\gep}$ which is the mollification of $X_t$. It appears when we consider the  conditional expectation of $M_{\gep}$
\begin{equation}
 X_{t,\gep}(y):= \int_{\bbR^d} X_t(y)\theta_{\gep}(x-y) \dd y= \bbE\left[ X_{\gep} \ | \ \cF_t \right].
\end{equation}
  The increments of the mollified fields also have a finite range dependence for the same reason and $(X_{\gep}-X_{s,\gep})(x)$ is independent of  $(X_{\gep}-X_{s,\gep})(y)$ and of  $X^{(s)}_{\cdot}(y)$ when $|x-y|\ge e^{-s}+2\gep$.
We let $K_{t,\gep}$ and $\bar K_{t,\gep,0}$ denote the covariance of $X_{t,\gep}$  and the cross-covariance with $\bar X_t$, 
\begin{equation}\label{crossover}\begin{split}
 K_{t,\gep}(x,y)&:=\bbE[X_{t,\gep}(x)X_{t,\gep}(y)]= \int_{\bbR^d} K_t(z_1,z_2)\theta_{\gep}(x-z_1 )\theta_{\gep}(x-z_1 ) \dd z_1 \dd z_2,\\ 
 \bar K_{t,\gep,0}(x,y)&:=\bbE[X_{t,\gep}(x)\bar X_{t}(y)]=\bbE[X_{\gep}(x)\bar X_{t}(y)]
 = \int_{\bbR^d} \overline K_t(z,y)\theta_{\gep}(x-z) \dd z.
 \end{split}
\end{equation}
Setting $\log_+ u= \max(\log u, 0)$, there exists a constant $C_R$ which is such that 
  \begin{equation}\label{thelogbound}
  \left| K_{t,\gep}(x,y) - t\wedge \log_+ \frac{1}{|x-y|\vee \gep} \right|\le C_R
  \end{equation}
The same bound are valid for $K_t$, $\bar K_t$ and $\bar K_{t,\gep,0}$, (with $\gep=0$ in the two first cases). For the upper bound on $\bar K_t$ no constant is needed and we have 
\begin{equation}\label{steak}
\bar K_t(x,y)\le t\wedge \log_+ \frac{1}{|x-y|}.
  \end{equation}
  The two bounds \eqref{thelogbound}-\eqref{steak} are easily obtained from the definition \eqref{crossover} and \eqref{defqt}. The reader can  refer to \cite[Appendix A.3]{critix} for a proof.
  
 \subsection{Uniform integrability via $L^1$ approximation  second moment estimates}

The following result is going to be helpful to prove uniform integrability for sequences that are not bounded in $L^2$.

 \begin{proposition}\label{metaprop}
 Consider $(Y_n)_{n\ge 0}$ a collection of positive random variables such that $\sup_{n\ge 0}\bbE[|Y_n|]<\infty$.
 Assume that there exists  
 $Y^{(r)}_n$ a sequence of approximation of $Y_n$, indexed by $r\ge 1$, which satisfies:
\begin{align*}
\mathrm{(A)} \qquad  &  \lim_{r\to \infty}  \limsup_{n\to \infty} \bbE\big[ |Y_n-Y^{(r)}_n|\big]=0\,  ; \\
 \mathrm{(B)}  \qquad  & \text{ For every $r\geq 0$ the sequence $(Y^{(r)}_n)_{n\ge 0}$ is uniformy integrable} 
\end{align*}
Then $(Y_t)_{t\ge 0}$ is uniformly integrable.
 \end{proposition}
\begin{proof}
We include a proof for completeness. 
Given $M>0$ let us set $A_{n,M}:=\{|Y_n|\ge M\}$. We need to show given that $\delta>0$ for $M$ sufficiently large 
\begin{equation}
\forall n \ge 0,  \ \bbE\left[ |Y_n|\ind_{A_{n,M}}\right]\le \delta
\end{equation}
We  have
\begin{equation}
 \bbE\left[ |Y_n|\ind_{A_{n,M}}\right]\le  \bbE\left[ |Y^{(r)}_n|\ind_{A_{n,M}}\right]+  \bbE\left[ |Y^{(r)}_n-Y_n|\right].
\end{equation}
Given $\delta>0$, using $(A)$ taking first $r=r(\delta)$ sufficiently large and then $n\ge n_0(\delta)$ we obtain that the second term is smaller than $\delta/2$.
 On the other hand, since the sequence $(Y_n)_{n\ge 1}$ is bounded in $L^1$ we have $P\left[ A_{n,M}\right] \le  C M^{-1}$. The uniform integrability of $Y^{(r)}_n$ thus implies 
 \begin{equation}
 \lim_{M\to \infty} \sup_{n\ge n_0} \bbE\left[ |Y^{(r)}_n|\ind_{A_{n,M}}\right]=0;
 \end{equation}
 and thus we can find $M_0(\delta,r)$ such that for all $n\ge n_0$
 $$ \bbE\left[ |Y^{(r)}_n|\ind_{A_{n,M_0}}\right]\le \delta/2$$ 
 Next we choose $M_1(\delta, n_0)$ such that 
\begin{equation}
 \forall n\in \lint 0,n_0-1\rint, \quad \bbE\left[ |Y_n|\ind_{A_{n,M}}\right]\le \delta.
\end{equation}
and conclude by setting  $M=\max(M_1,M_2)$.
\end{proof}

 \subsection{Truncation}
 In order to prove that $D_t$, $M_t$ and $M_{\gep}$ converge, we are going to considered \textit{truncated} versions of these processes where
high-values of the field $X_t$ are not taken into account.
Given $q,r\ge 1$ we define the following events
\begin{equation}\begin{split}\label{aqx}
 A^{(q)}_t(x)&:= \{ \forall s\in [0,t],\quad \bar X_s(x)< \sqrt{2d}s+q \},\\
    A^{(q,r)}_t(x)&:= \{ \forall s\in [0,t],\quad \bar X_s(x)< \sqrt{2d}s-\rho_r(s)+q \},\\
 \cA^{(q)}_R&:= \Big\{  \sup_{t>0}\sup_{|x|\le R }\left( \bar X_t(x)- \sqrt{2d} t \right)< q\Big\}=\bigcap_{|x|\le R} \bigcap_{t>0} A^{(q)}_t(x).
\end{split}\end{equation}
Setting $t_{\gep}:=\log (1/\gep)$, we define 

 \begin{equation}\begin{split}\label{trunk1}
  M^{(q)}_t(E)&:= \int_E  e^{\sqrt{2d}X_t(x)-dK_t(x)}\ind_{A^{(q)}_t(x)} \mu(\dd x) =: \int_E  W^{(q)}_t(x) \mu(\dd x),\\
   M^{(q)}_\gep(E)&:= \int_E  e^{\sqrt{2d}X_\gep(x)-dK_\gep(x)}\ind_{A^{(q)}_{t_{\gep}}(x)} \mu(\dd x) =: \int_E  W^{(q)}_{\gep}(x) \mu(\dd x),\\
  D^{(q)}_t(E)&:= \int_E (\sqrt{2d} t+q-\bar X_t(x)) W^{(q)}_{\gep}(x)  \mu(\dd x)=:  \int_E  Z^{(q)}_t(x) \mu(\dd x),
 \end{split}\end{equation}
 and
   \begin{equation}\begin{split}\label{trunk2}
 M^{(q,r)}_t(E)&:= \int_E  e^{\sqrt{2d}X_t(x)-d\bbE[X^2_t]}\ind_{A^{(q,r)}_t(x)}\mu(\dd x) =: \int_E  W^{(q,r)}_t(x) \mu(\dd x), \\  
 M^{(q,r)}_\gep(E)&:= \int_E  e^{\sqrt{2d}X_\gep(x)-d\bbE[X^2_{\gep}]}\ind_{A^{(q,r)}_{t_{\gep}}(x)}\mu(\dd x) =: \int_E  W^{(q,r)}_{\gep}(x) \mu(\dd x),\\
 D^{(q,r)}_t(E)&:= \int_E (\sqrt{2d} t+q-\rho_r(t)-\bar X_t(x)) W^{(q,r)}_t(x)\mu(\dd x)=: \int_E  Z^{(q,r)}_t(x) \mu(\dd x).
  \end{split} 
  \end{equation}
 Finally we set 
 \begin{equation}
  \overline{D}^{(q,r)}_{\infty}(E)=\limsup_{t\to \infty} D^{(q.r)}_t(E) \quad \text{ and }  \quad   \overline{D}^{(q)}_{\infty}(E)=\limsup_{t\to \infty} D^{(q)}_t(E).
 \end{equation}
We gather in a single proposition a collection of important information concerning the above defined processes.
The proofs are given in Appendix \ref{liistov}.

 \begin{proposition}\label{listov}
 The following holds:
 \begin{itemize}
  \item  [(i)] The processes  $(D^{(q,r)}_t)_{t\ge 0}$, $(M^{(q,r)}_t)_{t\ge 0}$ and $(M^{(q)}_t)_{t\ge 0}$ are nonnegative supermartingales for the filtration $(\cF_t)_{t\ge 0}$, $(D^{(q)}_t)_{t\ge 0}$ is a nonnegative martingale.
  \item [(ii)] For any $q\ge 1$, setting $\eta(q,r)= \bQ_q(\forall t\ge 0, \beta_t\ge \rho_r(t))$ we have
  \begin{equation}
   \lim_{t\to \infty}\bbE[D^{(q,r)}_t(E)]=\eta(q,r)\mu(E) 
  \end{equation}
and 
 \begin{equation}\label{rlimit}
 \begin{split}
  \lim_{r\to\infty} \limsup_{t\to \infty} \bbE[D^{(q)}_t(E) -D^{(q,r)}_t(E)]&=0,\\
    \lim_{r\to\infty} \limsup_{t\to \infty} \sqrt{t} \bbE[M^{(q)}_t(E) -M^{(q,r)}_t(E)]&=0,\\
      \lim_{r\to\infty} \limsup_{\gep\to 0} \sqrt{\log (1/\gep)} \bbE[M^{(q)}_{\gep}(E) -M^{(q,r)}_{\gep}(E)]&=0.
 \end{split}\end{equation}
 \item[(iii)] For any $q,r\ge 1$, we have 
 $$\lim_{t\to \infty}M_t=\lim_{t\to \infty} M^{(q)}_t= \lim_{t\to \infty} M^{(q,r)}_t=0.$$
 \item[(iv)] For any given $R>0$,
  $\lim_{q\to \infty}  \bP\left[\cA^{(q)}_R\right]=1.$
  As a consequence, there exists a random $q_0(R)\in \bbN\cup \{\infty\}$ which is almost surely finite and such that for every $q\ge q_0$, $t\ge 0$, $\gep\in(0,1]$ and $E\subset B(0,R)$
 \begin{equation}
  D^{(q)}_t(E)=(D_t+q M_t)(E)   , \quad  M^{(q)}_t(E)= M_t(E) \quad \text{ and } \quad  M^{(q)}_\gep(E)= M_\gep(E).
  \end{equation}
  \item[(v)]Almost surely,  for all $q\ge q_0(R)$ (defined by $(iv)$) and  $E\subset B(0,R)$ we have
  $$ \bar D^{(q)}_{\infty}(E)= \bar D_{\infty}(E).$$
 \end{itemize}

\end{proposition}

%

 \subsection{Moment estimates}
 
 The following technical estimates are necessary to compute the second moments of $M^{(q,r)}_t$ and $M^{(q,r)}_{\gep}$ (which play and important role in our proof). 
For $E\in \cB_b$  we let $\Diam(E)$ denote the diameter for the $\ell_{\infty}$ norm that is 
$$\Diam(E)= \sup\{ x,y\in E, \max_{i\in \lint 1,d\rint} |x_i-y_i|\}.$$

\begin{proposition}\label{lateknik}
Setting  $u(x,y,t) :=  \left(\log \frac{1}{|x-y|\wedge 1}\right)   \wedge  t$. There exists $C_{q,r,R}>0$ such that  forall $t\ge 0$ and  $x,y\in B(0,R)$,
 \begin{equation}\label{tek1}
 \  \bbE\left[ W^{(q,r)}_t(x) W^{(q,r)}_t(y)\right] \le C e^{d u-\sqrt{2d}\rho(u)} (u+1)^{-3/2}  (t-u+1)^{-1}.
 \end{equation}
  Setting $v(x,y,\gep):= \left(\log \frac{1}{|x-y|\wedge 1}\right) \wedge \log( 1/\gep)$, there exists $C_{q,R}>0$ such that for all $\gep\in (0,1]$ and  $x,y\in B(0,R),$
 \begin{equation}\label{tek2}
  \  \bbE\left[ W^{(q,r)}_\gep(x) W^{(q,r)}_\gep(y)\right] \le C e^{d v-\sqrt{2d}\rho(u)} (v+1)^{-3/2}  ( t_{\gep}-v+1)^{-1}.
 \end{equation}
 As  consequences we have (with a possibly different constant $C$)
 \begin{itemize}
  \item[(i)] Given $R>0$, there exists a finite measure $\nu_R$ on $B(0,R)$  such that  forall  $E\subset B(0,R)$ with $\Diam(E)\le 1$ we have

  \begin{equation}\label{hiphip}
  \limsup_{t\to \infty} \bbE\left[ t(M^{(q,r)}_t(E))^2 \right]\le C e^{-(\sqrt{2d}-1)\rho\left(\log (\frac 1 {\Diam(E)})\right)}\nu_R(E).
  \end{equation}
  \item[(ii)] For any $E\in \mathcal B_b$, we have
  \begin{equation}\label{whendelta}
  \begin{split}
  &\lim_{\delta\to 0}   \sup_{t\ge 0} \int_{E^2} \bbE\left[ t W^{(q,r)}_t(x)W^{(q,r)}_t(y)\right]\ind_{\{|x-y|\le \delta\}} \mu(\dd x)\mu( \dd y)=0,\\
&\lim_{\delta\to 0}   \sup_{\gep\in (0,1]} \int_{E^2} \bbE\left[ t_{\gep} W^{(q,r)}_\gep(x)W^{(q,r)}_\gep(y)\right]\ind_{\{|x-y|\le \delta\}} \mu(\dd x)\mu( \dd y)=0.
\end{split}
  \end{equation}
 \end{itemize}
\end{proposition}
 
 \subsection{Weak convergence of measure}
 
 Let us finally recall how convergence of $M_t$ and $M_{\gep}$ seen as processes indexed by $\mathcal B_b$ implies weak convergence of measure
 (for a proof the reader can  refer to \cite[Proposition B.1]{critix}).
 
 \begin{proposition}\label{weako}
 Let $M_n$ be a sequence of non-negative random measures and let us assume that for any fixed $E\in \cB_b$, 
 $M_n(E)$ converges in probability towards a finite limit $\bar M(E)$.
 Then the following holds
 \begin{itemize}
  \item [(i)] For any $f\in B_b$, $M_n(f)$ converges in probability towards a finite limit $\bar M(f)$.
  \item [(ii)]
  There exists a random measure $M$  towards which $M_n$ converges in probability.
  \item [(iii)] If additionally there exist $K>0$ such that $\bbE[M_n(E)]\le K\gl(E)$ for all $E\in \mathcal B_b$ and $n$, then for all    $f\in B_b$, $\bbP[\bar M(f)= M(f)]=1$.
 \end{itemize}
If the convergence of  $M_n(E)$ holds a.s.\ then the convergence $(i)$ and $(ii)$ also hold a.s.

\end{proposition}

\section{Convergence of the derivative martingale} \label{caseofdt}

 \subsection{The main statement}
The convergence of $D_t$ is going to be deduced from that of  $D^{(q,r)}_t$ using of the technical lemmas from Section \ref{prelimlim}.
The main statement to be proved in this section is thus the following.

 \begin{proposition}\label{troncR}
 Under the assumption \eqref{clazomp}, for any $q,r\ge 0$,  for every $E\in \mathcal B_b$
$(D^{(q,r)}_t(E))_{t\ge 0}$  is bounded in $L^2$. As a consequence we have
 \begin{equation}\label{toctoctoc}
   \bbE\left[\overline D^{(q.r)}_\infty(E) \right]= \eta(q,r) \mu(E),
\end{equation}
where 
$$\eta(q,r):= q\bQ_q\left( \forall t>0, \beta_t \ge \rho_r(t) \right)>0.$$
Furthermore  $D^{(q,r)}_t$ converges almost surely weakly to a random measure $ D^{(q,r)}_{\infty}$ which is such that for every $E\in \mathcal B_b$
\begin{equation}\label{exprq}
 \bbP\left[D^{(q,r)}_\infty(E)=  \overline D^{(q.r)}_\infty(E)\right]=1.
\end{equation}
Lastly given  $\alpha\in (0,\sqrt{2d}-1)$ and  $R>0$ we have
\begin{equation}\label{unionofsets}
 \sup_{E\subset B(0,R)}D^{(q,r)}_{\infty}(E)e^{\alpha\rho(\log (1/\Diam(E))}<\infty,
\end{equation}
(in particular the measure is atomless).
 \end{proposition}

Let us explain how this section is organized.
The convergence of $D_t$ (Theorem \ref{derivmartin}) is not very hard to derive from the above proposition. 
We perform this derivation in Section \ref{tekel}. In Section \ref{tekteklem}, we present a relation between $D^{(q,r)}_t$ and $M^{(q,r)}_t$ obtained using conditional expectation. This relation is helpful (but not at all crucial) to compute the second moment of $D^{(q,r)}_t$ and this is the reason one it is introduced in this section. Its principal 
use (as will be seen in Section \ref{caseofmt}) is for the proof of the convergence of $\sqrt{\frac{\pi t}{2}}M^{(q,r)}_t$.
Finally in Section \ref{thundertem} we prove Proposition \ref{troncR}.

\subsection{Proof of Theorem \ref{derivmartin}}\label{tekel}

The proof goes in two steps. Firstly, we use Proposition \ref{troncR} in order to deduce convergence of $D^{(q)}_{\infty}$.

\begin{proposition}\label{uint}
For any $q\ge 0$, for any $E\in \mathcal B_b$, 
 $(D^{(q)}_t(E))_{t\ge 0}$ is uniformly integrable and hence
 \begin{equation}\label{despe}
\bbE\left[ \overline  D^{(q)}_{\infty}(E)\right]=q\mu(E).
\end{equation}
Furthermore, there exists a locally finite atomless measure $D^{(q)}_{\infty}$ such that  
for any $E\in \mathcal B_b$, we have with probability $1$
\begin{equation}\label{expq}
\bbP\left[ D^{(q)}_{\infty}(E)=  \bar D^{(q)}_{\infty}(E)\right]=1. 
\end{equation}

\end{proposition}

\noindent Secondly, we deduce from Proposition \ref{uint} the convergence of $D_t$ and prove that the limit satisfies all the stated properties.
We  let $D_\infty$ denote the measure defined by (the limit is well defined by monotonicity)
\begin{equation}
D_{\infty}(E):= \lim_{q\to \infty }D^{(q)}_{\infty}(E).
\end{equation}
With this definition Theorem \ref{derivmartin} can be restated in the following manner.

\begin{proposition}\label{conclu}

The measure $D_\infty$ is a.s.\ locally finite and atomless  and $D_t$ converges weakly to $D_{\infty}$.
We have for any $E\in \mathcal B_b$
\begin{equation}\label{topitope}
\bbP\left[ D_{\infty}(E)=  \bar D_{\infty}(E)\right]=1. 
\end{equation}
and $\bbE[ D_{\infty}(E)]=\infty$ if $\mu(E)>0$.
Furthermore $D_{\infty}$ is atomless and its topological support a.s.\ coincides with that of $\mu$.
\end{proposition}
\noindent We now prove the two above propositions in the order in which they were introduced.

\begin{proof}[Proof of Proposition \ref{uint}]
 
 The uniform integrability of $D^{(q)}_{\infty}(E)$ follows from Proposition \ref{metaprop}. 
 The assumption $(A)$ of Proposition \ref{metaprop} is satisfied using \eqref{rlimit} while $(B)$ (uniform integrability in $t$ of $D^{(q,r)}_t$) is simply given by Proposition \ref{troncR} .
 The weak convergence of $D^{(q)}_t$ can then be deduce using Proposition \ref{weako}.
Note that from \eqref{exprq}, \eqref{expq}, for any $E\in \mathcal B_b$ we have
\begin{equation}
 \bbE\left[D^{(q)}_{\infty}(E)-D^{(q,r)}_\infty(E)\right]= \left(q-\eta(q,r)\right)\mu(E).
\end{equation}
Now since $D^{(q,r)}_\infty(E)$ is increasing in $r$ we obtain the following convergence in $L^1$
\begin{equation}
 D^{(q)}_{\infty}(E)=\lim_{r\to \infty} D^{(q,r)}_{\infty}(E)
\end{equation}
This implies (cf. Proposition \ref{weako})  a.s.\ weak convergence of the measure $D^{(q)}_{\infty}=\lim_{r\to \infty} D^{(q,r)}_{\infty}$ 
and in particular (from \eqref{unionofsets}) that $D^{(q)}_{\infty}$ is atomless.
\end{proof}

\begin{proof}[Proof of Proposition \ref{conclu}]
From item $(v)$ of Proposition \ref{listov}, we obtain that $\overline D_\infty(E)=\lim_{q\to \infty}\overline D^{(q)}_{\infty}(E)$, and that the limit is finite (since the sequence is stationary in $q$). Atomlessness is a consequence of that of $D^{(q)}_{\infty}$.
Combining \eqref{expq}
and Proposition \ref{weako}, we obtain that $D_t$ converges weakly to $D_{\infty}$ as a measure and that \eqref{topitope} holds.
The fact that $\bbE[ D_{\infty}(E)]=\infty$ if $\mu(E)>0$ then follows from \eqref{despe}.

\medskip

The final point is to check that the topological support of $D_{\infty}$ coincides with that of $\mu$.
We simply have to check that if $E\in \mathcal B_b$ is such that $\mu(E)>0$ then $\bbP\left[ D_{\infty}(E)>0\right]=1$
and apply this to a countable base of open sets to conclude.

\medskip

\noindent We need only to show that $\bbP\left[ D_{\infty}(E)>0\right]\in\{0,1\}$ since we already know that $$\bbP\left[ D_{\infty}(E)>0\right]>0$$ ($D_{\infty}(E)$ has infinite expectation). 
This $0-1$ law is a consequence of the fact that the event belongs to the tail $\sigma$-algebra (it is independent of $\mathcal F_s$ for all $s>0$). This can be deduced from the fact (which follows from straightfoward comparisons of integrals) that 
\begin{equation}
 D_{\infty}(E)>0   \quad \Leftrightarrow  \quad \lim_{t\to \infty} \int_E  e^{\sqrt{2d}X^{(s)}_t(x)-d t} \mu(\dd x)>0,
\end{equation}
and the event on the right-hand side is independent of $\mathcal F_s$.
\end{proof}

\subsection{A technical observation concerning conditional expectation}\label{tekteklem}

We state and prove now an important result relating $M^{(q,r)}_t$ to $D^{(q,r)}_s$.
To write the exact relation we need to introduce a slightly modified version of $D^{(q,r)}_t$, by setting (for $E\in \cB_b$). Recalling the notations introduced in Section \ref{gbesti}
\begin{equation}\begin{split}
 \tilde Z^{(q,r)}_s(x)&:= Z^{(q,r)}_s(x) \bQ_{\sqrt{2d} s+q-\rho_{r}(s)- \bar X_s}\left[ \forall u\ge 0,  \beta_u \ge \rho_{r+s}(u) \right],\\
 \tilde D^{(q,r)}_s(E)&:= \int_E   \tilde Z^{(q,r)}_s(x)\mu(\dd x)
\end{split}\end{equation}
Setting $\delta(u):= \bP_0\left[\exists t\ge 0, \beta_t\ge \rho_u(t)  \right]$, the following inequalities are valid for any $x\in \bbR^d$ and 
$E\in \mathcal B_b$,
 \begin{equation}\begin{split}\label{cook}
  (1-\delta(r+s)) Z^{(q,r)}_s(E)&\le  \tilde Z^{(q,r)}_s(E) \le   Z^{(q,r)}_s(E),\\
  (1-\delta(r+s)) D^{(q,r)}_s(E)&\le  \tilde D^{(q,r)}_s(E) \le   D^{(q,r)}_s(E).
  \end{split}
\end{equation}
The second line in \eqref{cook} can be obtained from the first by integrating. For the first one, we simply use the fact that the following probability is decreasing in $a$,
\begin{equation}
 \bQ_{a}\left[ \exists  u\ge 0,  \beta_u \ge \rho_{r}(u) \right].
\end{equation}

   \begin{lemma}\label{lecondit}
 Given $E\in \cB_b,$ for any fixed $s$ the following convergence holds  in $L^2$.
 \begin{equation}\label{wohop}
\lim_{t\to \infty} \bbE\left[ \sqrt{\frac{\pi (t-s)}{2}}  M^{(q,r)}_t(E) \ | \ \cF_s\right]=  \tilde D^{(q,r)}_s(E)
 \end{equation}

\end{lemma}

\begin{proof}
 It is sufficient to show that for any $x\in E$ we almost surely have
 \begin{equation}\label{gross}
 \lim_{t\to \infty} \bbE\left[ \sqrt{\frac{\pi (t-s)}{2}}  W^{(q,r)}_t(x) \ | \ \cF_s\right]=\tilde Z^{(q,r)}_s(x),
 \end{equation}
and that
 \begin{equation}\label{gross2}
  \bbE\left[ \sqrt{\frac{\pi (t-s)}{2}}  W^{(q,r)}_t(x) \ | \ \cF_s\right]\le  Z^{(q,r)}_s(x).
 \end{equation}
Indeed using  the dominated convergence (\eqref{gross2} being used for the domination), we can deduce deduce 
from \eqref{gross} that the convergence \eqref{wohop} holds in the almost-sure sense.
For the convergence in $L^2$,
we can use dominated convergence again (with $(Z^{(q,r)})^2$ being used for domination) to obtain
\begin{equation}
 \lim_{t\to \infty}\bbE\left[  \left(\tilde Z^{(q,r)}_s(E)- \bbE\left[ \sqrt{\frac{\pi (t-s)}{2}}  M^{(q,r)}_t(E) \ | \ \cF_s\right]\right)^2 \right]=0.
 \end{equation}
Let us now proceed with the proof of \eqref{gross} and \eqref{gross2}.
Omitting the dependence in $x$  for readability, factoring out  $W^{(q,r)}_s$ (which is $\cF_s$ measurable) and 
then using Cameron-Martin's formula for the Brownian Motion $(X^{(s)}_{\cdot}(x))$  we obtain that
\begin{equation}\label{laconvmone}\begin{split}
  \bbE\left[  W^{(q,r)}_t \ | \ \cF_s\right]&= W^{(q,r)}_s  \bE\left[e^{\sqrt{2d}X^{(s)}_{t-s}-d (t-s)}\ind_{\{ \forall u\in [0,t-s], \ X^{(s)}_{u} 
  \le  \sqrt{2d}(s+u)- \rho_{r}(s+u)+q-\bar X_s\}} 
   \ | \ \cF_s \right]\\
\\&=W^{(q,r)}_s \bP\left[  \sup_{u\in [0,t-s]} X^{(s)}_{u} \le  \sqrt{2d} s+q-\rho_{r}(s+u)- \bar X_s  \ | \ \cF_s \right]
\\&=W^{(q,r)}_s \bP_{ \sqrt{2d} s+q-\rho_{r}(s)- \bar X_s}\left[ \forall u\in [0,t-s],\ B_u\ge \rho_{r+s}(u)  \right].
\end{split}\end{equation}
To conclude, the proof of \eqref{gross}-\eqref{gross2}, we apply item $(i)$ from Lemma \ref{labesse} with $a=\sqrt{2d} s+q-\rho_{r}(s)- \bar X_s$ (and replacing $t$ by $t-s$).

\end{proof}

 \subsection{Proof of Proposition \ref{troncR}}\label{thundertem}
We take $s$ sufficiently large so that $\delta(s)\ge 1/2$.
Hence we have
\begin{multline}
\bbE[ D^{(q,r)}_s(E)^2]\le 4 \bbE\left[\tilde D^{(q,r)}_s(E)^2\right]\\
\le \lim_{t\to \infty} 2\pi t \bbE\left[\bbE\left[ M^{(q,r)}_t(E) \ | \ \cF_s \right]^2\right]
\le\limsup_{t\to \infty}2\pi t \bbE\left[ M^{(q,r)}_t(E)^2\right].
\end{multline}
Using Proposition \ref{lateknik} (item $(i)$ in the consequences) we obtain that for any $E\in B(0,R)$
\begin{equation}\label{trook}
 \bbE[ \bar D^{(q,r)}_\infty(E)^2]\le  \limsup_{s\to \infty} \bbE[ D^{(q,r)}_s(E)^2]\le \nu_R(B(0,R)) e^{-(\sqrt{2d}-1) \rho(1/\Diam(E))}<\infty
\end{equation}
This bound on the second moment implies uniform integrability. Hence using Proposition \ref{listov} item $(ii)$ we obtain 
\begin{equation}
  \bbE[ \bar D^{(q,r)}_\infty(E)]=\lim_{t\to \infty} \bbE\left[D^{(q,r)}_t(E) \right]=\mu(E)\eta(q,r).
\end{equation}
Finally to check that \eqref{unionofsets} holds, it is sufficient to show that the supremum over 
all dyadic cubes of the form $E_{n,x}=2^{-n}([0,1]^d+x)$ with $E_{n,x}\subset  B(0,2 R)$ is finite.
We have 
\begin{equation}\begin{split}
\sum_{n\ge 1}\bbP\Big[\exists x\in  \bbZ^d,& E_{n,x}\subset  B(0,2 R) \text{ and }  D^{(q,r)}_{\infty}(E_{n,x})\ge e^{-\alpha \rho\left(n\log 2\right)}  \Big]\\
&\le \sum_{n\ge 1} \sum_{\{x\in\bbZ^d  :  E_{n,x}\subset  B(0,2 R)\}}   \bbP\left[  D^{(q,r)}_{\infty}(E_{n,x})\ge e^{-\alpha \rho\left( n\log 2\right)}  \right]\\
&\le \sum_{n\ge 1} \sum_{\{x\in\bbZ^d  :  E_{n,x}\subset  B(0,2 R)\}}  e^{\alpha \rho\left( n\log 2\right)}  \bbP\left[  D^{(q,r)}_{\infty}(E_{n,x})^2\right]\\
&\le \sum_{n\ge 1} \sum_{\{x\in\bbZ^d  :  E_{n,x}\subset  B(0,2 R)\}}  e^{-(\sqrt{2}-1-\alpha) \rho\left( n\log 2\right)} \nu_{2R}(E_{n,x})\\
&\le \nu_{2R}(B(0,2R))\sum_{n\ge 1} e^{-(\sqrt{2}-1-\alpha) \rho\left( n\log 2\right)}.
\end{split}
\end{equation}
In the penultimate inequality we have used \eqref{trook}.
In the last line we have used the assumption that $\nu_{2R}$ is a finite measure, and that $\rho(u)\ge u^{1/4}$ for large values of $u$ (an assumption made at the very beginning of Section \ref{prelimlim}  and that we  use only here).
\qed

\section{The convergence of $M_t$} \label{caseofmt}

\noindent The main statement proved in this section is the following.
\begin{proposition}\label{dor}
For any $E\in \mathcal B_b$, we have the following convergence in $L^2$ for every $q,r\ge 1$
 \begin{equation}\label{potipoti}
 \lim_{t\to \infty} \sqrt{\frac{\pi t}{2}} M^{(q,r)}_t(E)= \bar D^{(q,r)}_{\infty}(E).
 \end{equation}
 As a consequence, the following convergence holds in $L^1$ for every $q\ge 1$
\begin{equation}\label{convtroncq}
   \lim_{t\to \infty} \sqrt{\frac{\pi t}{2}} M^{(q)}_t(E)= \bar D^{(q)}_{\infty}(E),
 \end{equation}
and the following convergence in probability
\begin{equation}\label{finally}
    \lim_{t\to \infty} \sqrt{\frac{\pi t}{2}} M_t(E)= \bar D_{\infty}(E).
\end{equation}

\end{proposition}

\begin{proof}[Proof of Theorem \ref{critmartin}]
 As a consequence of \eqref{convtroncq}, we obtain from Proposition \ref{weako} that $M^{(q)}_t$ converges weakly in probability towards $D^{(q)}_{\infty}$. Since - by Proposition \ref{listov} item (iv) - given $R>0$, for $q$ sufficiently large the respective restrictions of $M_t$  and $D_{\infty}$ to  $B(0,R)$ coincide with that of $M_t^{(q)}$ and  $D^{(q)}_\infty$  (for all $t$), $M_t$ converges weakly to $D_{\infty}$.
\end{proof}

\begin{proof}[Proof of Proposition \ref{dor}] Let us start with the proof of \eqref{potipoti}. We drop $E$ from the notation for better readability.
 Since $D^{(q,r)}_s$ converge in $L^2$, it is sufficient to prove that 
 \begin{equation}
 \lim_{s\to \infty} \limsup_{t\to \infty}\bbE\left[  \left|\sqrt{\frac{\pi t}{2}} M^{(q,r)}_t-D^{(q,r)}_s\right|^2 \right]=0.
 \end{equation}
Using the conditional expectation to make an orthogonal decomposition in $L^2$ we have
\begin{multline}\label{ladekomp}
 \bbE\left[  \left|\sqrt{\frac{\pi t}{2}} M^{(q,r)}_t-D^{(q,r)}_s\right|^2 \right]\\=
 \frac{\pi t}{2} \bbE\left[  \left| M^{(q,r)}_t-\bbE\left[  M^{(q,r)}_t \ | \ \cF_s \right]\right|^2 \right]
+ \bbE \! \left[  \! \left( \!
\bbE \! \left[ \! \sqrt{\frac{\pi t}{2}} M^{(q,r)}_t \ | \ \cF_s\right]- D^{(q,r)}_s \!  \right)^2 \! \right].
\end{multline}
From Lemma \ref{lecondit}, we have 
\begin{equation}
\lim_{s\to \infty}\lim_{t\to \infty}  \bbE \! \left[  \! \left( \!
\bbE \! \left[ \! \sqrt{\frac{\pi t}{2}} M^{(q,r)}_t \ | \ \cF_s\right]- D^{(q,r)}_s \!  \right)^2 \! \right]\\
=\lim_{s\to \infty} \bbE \! \left[ \left(\! \tilde D^{(q,r)}_s- D^{(q,r)}_s \!  \right)^2 \! \right]=0.
\end{equation}
To control the first term,  we set  (recall \eqref{trunk2})
$$  \xi_{s,t}(x):= W^{(q,r)}_{t}(x)- \bbE\left[ W^{(q,r)}_{t}(x) \ | \ \cF_s \right].$$
Expanding the square we obtain
\begin{equation}\label{xixi}
\bbE\left[  \left| M^{(q,r)}_t-\bbE\left[  M^{(q,r)}_t \ | \ \cF_s \right]\right|^2 \right] =\int_{E^2} \bbE\left[ \xi_{s,t}(x) \xi_{s,t}(y)   \right] \mu(\dd x) \mu(\dd y).
\end{equation}
Note that $\xi_{s,t}(x)$ is measurable with respect to $\cF_s\vee \sigma( (X^{(s)}_u(x))_{u\ge 0})$. If $|x-y|\ge e^{-s}$ then  $X^{(s)}_{\cdot}(x)$ and $X^{(s)}_{\cdot}(y)$  are independent and independent of $\cF_s$. This implies  conditional independence of $\xi_{s,t}(x)$ and  $\xi_{s,t}(y)$ and hence
\begin{equation}\label{czer}
\bbE\left[ \xi_{s,t}(x) \xi_{s,t}(y) \ | \ \cF_s  \right]= \bbE\left[ \xi_{s,t}(x) \ | \ \cF_s  \right] \bbE\left[ \xi_{s,t}(y)  \ | \ \cF_s \right]=0.
\end{equation}
On the other hand when $|x-y|\le e^{-s}$, we have 
\begin{equation}\label{cpzer}
 \bbE\left[ \xi_{s,t}(x) \xi_{s,t}(y)\right]\le 
  \bbE\left[ W^{(q,r)}_{t}(x) W^{(q,r)}_{t}(y)\right].
\end{equation}
Thus combining \eqref{czer} and \eqref{cpzer} we have
\begin{multline}\label{czer3}
\bbE\left[  \left| M^{(q,r)}_t-\bbE\left[  M^{(q,r)}_t \ | \ \cF_s \right]\right|^2 \right]\\
\le \int_{E^2}  \bbE\left[ W^{(q,r)}_{t}(x) W^{(q,r)}_{t}(y)   \right]  \ind_{\{|x-y|\le e^{-s}\} } \mu(\dd x)\mu( \dd y).
\end{multline}
and we can conclude using Proposition \ref{lateknik} (more precisely \eqref{whendelta}) that 
\begin{equation}
\lim_{s\to \infty} \limsup_{t\to \infty}\bbE\left[  \left| M^{(q,r)}_t-\bbE\left[  M^{(q,r)}_t \ | \ \cF_s \right]\right|^2 \right]=0.
\end{equation}
Let us now prove \eqref{convtroncq}. Removing $E$ from the equation for the sake of readability, we have
\begin{multline}
 \bbE\left[ \left| \sqrt{\frac{\pi t}{2}} M^{(q)}_t- \bar D^{(q)}_{\infty} \right|\right]=
\sqrt{\frac{\pi t}{2}}  \bbE\left[  M^{(q)}_t- \bar M^{(q,r)}_{\infty} \right]
\\ 
+\bbE\left[ \left| \sqrt{\frac{\pi t}{2}} M^{(q,r)}_t- \bar D^{(q,r)}_{\infty} \right|\right]+ \bbE\left[ \left| \bar D^{(q)}_\infty- \bar D^{(q,r)}_{\infty} \right|\right]
\end{multline}
Using  \eqref{rlimit} (from  Proposition \ref{listov}) we see that the $\limsup$ for the first and third term (using Fatou) can be made arbitrarily small by taking $r$ large, while the second term tends to zero from \eqref{potipoti}.
Finally \eqref{finally} can be obtained from \eqref{convtroncq} simply using Proposition \ref{listov} (iv)-(v).

\end{proof}

\section{The convergence of $M_{\gep}$}\label{caseofmgep}

\subsection{The main statement and its proof}

\noindent The strategy of the previous section can be adapted to prove the convergence of $M^{}_{\gep}$. 
We  show that $\sqrt{{\pi \log(1/\gep)}/{2}} M^{(q)}_\gep(E)$ converges to the same limit as $D^{(q)}_t(E)$ and $M^{(q)}_t(E)$.

\begin{proposition}\label{gepqgep}
 For any $E\in \mathcal B_b$ and $q\ge 0$, `we have the following convergence in $L_2$ for every $q,r\ge 1$
 \begin{equation}
 \lim_{\gep\to 0} \sqrt{\frac{\pi \log(1/\gep)}{2}} M^{(q,r)}_\gep(E)= \bar D^{(q,r)}_{\infty}(E).
 \end{equation}
 As a consequence the following convergence holds for every $q$ in $L^1$
 \begin{equation}
   \lim_{\gep\to 0} \sqrt{\frac{\pi \log(1/\gep)}{2}} M^{(q)}_\gep(E)=\bar D^{(q)}_{\infty}(E),
 \end{equation}
and the following convergence in probability
\begin{equation}\label{finally2}
    \lim_{\gep\to 0} \sqrt{\frac{\pi t}{2}} M_t(E)= \bar D_{\infty}(E).
\end{equation}
\end{proposition}

\noindent We can deduce our main result from the above proposition.

 \begin{proof}[Proof of Theorem \ref{mainres}]
 We can simply repeat the argument used in the previous section for the proof of Theorem \ref{critmartin}.
 \end{proof}

\noindent The proof of Proposition \ref{gepqgep} follows the same structure as that of Proposition \ref{dor}.
The crucial point is to adapt Lemma \ref{lecondit} in the case where $X_t$ is replaced by $X_{\gep}$. The result below is proved in the next subsection.

\begin{lemma}\label{lecondit2}
  For any fixed $s\ge 0$ we have the following convergence in $L^2$.
 \begin{equation}
 \lim_{\gep \to 0} \bbE\left[ \sqrt{\frac{\pi (\log (1/\gep))}{2}} M^{(q,r)}_\gep(E) \ | \ \cF_s\right]=   \tilde D^{(q,r)}_s(E).
 \end{equation}
 
\end{lemma}

\begin{proof}[Proof of Proposition \ref{gepqgep}]
 Like  Proposition \ref{dor}, and 
 it is sufficient to prove that 
 \begin{equation}
 \lim_{s\to \infty} \limsup_{\gep\to 0}\bbE\left[  \left|\sqrt{\frac{\pi t_{\gep}}{2}} M^{(q,r)}_\gep-D^{(q,r)}_s\right|^2 \right]=0.
 \end{equation}
 Using a decomposition analogous to  \eqref{ladekomp} and Lemma \ref{lecondit2}, we only need to prove the following 
 \begin{equation}\label{lllast}
 \lim_{s\to \infty} \limsup_{\gep\to 0} t_{\gep} \bbE\left[  \left| M^{(q,r)}_\gep-\bbE\left[  M^{(q,r)}_\gep  \ | \ \cF_s \right] \right|^2 \right]=0.
 \end{equation}
Setting $
 \xi_{s,\gep}(x) := W^{(q,r)}_{\gep}(x)-\bbE[W^{(q,r)}_{\gep}(x) \ | \ \cF_s]$ (recall \eqref{trunk2})
we have like for \eqref{xixi}
\begin{equation}\label{xixi2}
\bbE\left[  \left| M^{(q,r)}_\gep-\bbE\left[  M^{(q,r)}_\gep \ | \ \cF_s \right]\right|^2 \right] 
=\int_{E^2} \bbE\left[ \xi_{s,\gep}(x) \xi_{s,\gep}(y)   \right] \dd x \dd y.
\end{equation}
Recalling the observations of Section \ref{considerations}, 
$(X^{(s)}_{\cdot}(x), (X_{\gep}-X_{\gep,s})(x) )$ and  $(X^{(s)}_{\cdot}(y), (X_{\gep}-X_{\gep,s})(y))$ are independent and independent of $\cF_s$ whenever
$|x-y|\ge e^{-s}+2\gep$. Hence  
 $\xi_{s,\gep}(x) $ and $\xi_{s,\gep}(y) $ are conditionally independent given $\cF_s$ and   like for \eqref{czer} we have
$$ |x-y|\ge e^{-s}+2\gep \quad \Rightarrow  \quad  \bbE\left[\xi_{s,\gep}(x) \xi_{s,\gep}(y) \right]=0.$$
Proceeding as in  \eqref{czer3}, this implies 
 \begin{equation*}
\bbE\left[  \! \left| M^{(q,r)}_\gep-\bbE\left[  M^{(q,r)}_\gep \ | \ \cF_s \right]\right|^2  \! \right]
  \le \int_{E^2} \!\! \!  \bbE\left[ W^{(q,r)}_{\gep}(x)  W^{(q,r)}_{\gep}(y) \right] \ind_{\{|x-y|\le e^{-s}+2\gep\}} \mu(\dd x) \mu(\dd y)
 \end{equation*}
and we conclude that \eqref{lllast} holds using Proposition \ref{lateknik}.

\end{proof}

\subsection{Proof of Lemma \ref{lecondit2}}

The proof of the $L^2$ convergence follows the same spirit as that of Lemma \ref{lecondit} but is more technical.
Using the convention $t=t_{\gep}=\log( 1/\gep)$, we need to show that
\begin{equation}\label{gregre}
 \lim_{\gep \to 0} \sup_{x\in E} \bbE\left[ \left( \tilde Z^{(q,r)}_s(x)- \bbE\left[ \sqrt{\frac{\pi t}{2}} W^{(q)}_\gep(x) \ | \ \cF_s\right]\right)^2 \right]=0.
\end{equation}
The uniformity in \eqref{gregre} implies that the convergence in $L^2$ is maintained after integrating with respect to $x$ over $E$. Taking $\cF_s$-measurable 
terms out of the expectation we obtain
\begin{equation}
   \bbE[  \sqrt{{\pi t}/{2}}  W^{(q,r)}_\gep(x) \ | \ \cF_s ] = W^{(q,r)}_{s,\gep}(x) \times  V^{(q,r)}_{s,\gep}(x),
 \end{equation}
 where
 \begin{equation*}\begin{split}
 W^{(q,r)}_{s,\gep}(x)&:= e^{\sqrt{2d}X_{s,\gep}(x)-d K_{s,\gep}(x)} \ind_{A^{(q,r)}_s(x)},\\
  V^{(q,r)}_{s,\gep}(x)&:=   \sqrt{\frac{\pi t }{2}} \bbE\left[ e^{\sqrt{2d}(X_{\gep}-X_{s,\gep})(x)- d(K_{\gep}-K_{s,\gep})(x)}
  \ind_{\{\forall u\in [s,t],\  \bar X_u(x)\le q+ \sqrt{2d}u-\rho_r(t) \}   } \ | \ \cF_s \right].
   \end{split}\end{equation*}
To prove  \eqref{gregre}, setting 
\begin{equation*}
\zeta_s(x):=(q+\sqrt{2d}s-\rho_r(s)- \bar X_s(x))_+\bP_{q+\sqrt{2d}s-\rho_r(s)- \bar X_s(x)}\left[ \forall u\ge0, \ \beta_u\ge \rho_{s+r}(u)\right] 
\end{equation*}
it is sufficient to prove that the two following convergences holds and
that the respective limit are uniformly bounded in $L^4$ \footnote{This claim is backed by the following inequality
\begin{equation}
 \bbE[ (A_tB_t-AB)^2]^2\le 4 \left(\bbE[(A_t-A)^4]\bbE[B^4_t] +  \bbE[(B_t-B)^4]\bbE[ A^4] \right)
\end{equation}}
\begin{equation}\label{unil4}\begin{split}
 \lim_{\gep \to 0}\sup_{x\in E} \ & \bbE\left[ (W^{(q,r)}_{s,\gep}(x)-W^{(q,r)}_s(x))^4\right]=0,\\
 \lim_{\gep \to 0}\sup_{x\in E} \ & \bbE\left[ (V^{(q,r)}_{s,\gep}(x)- \zeta_s(x))^4 \right]=0.
  \end{split}
\end{equation}
The first line in \eqref{unil4} follows from the the uniform convergence of $K_{s,\gep}(x,y)$  and $K_{s,\gep,0}(x,y)$  towards $K_{s}(x,y)$, which implies that
\begin{equation}
 \lim_{\gep\to 0} \sup_{x\in E}\bbE\left[  \left(e^{\sqrt{2d}X_{s,\gep}(x)-d K_{s,\gep}(x)}- e^{\sqrt{2d}X_{s}(x)-d K_{s}(x)}\right)^4 \right]=0.
\end{equation}
For the second  line in \eqref{unil4}, by translation invariance of $\bar X$, the quantity in the supremum does not depend on $x$. 
Thus we do not need to worry about the $\sup$. We omit the dependence in  $x$ in the next computations for better readability.  We rewrite the event appearing in the definition of   $V^{(q)}_{s,\gep}$ as
$$\{\forall u\in [0,t-s],\   X^{(s)}_u\le q+ \sqrt{2d}(s+u)- \bar X_s-\rho_r(s+u) \}.$$
Using Cameron-Martin Formula 
the exponential tilt in  $V^{(q,r)}_{s,\gep}$  has the effect of shifting the mean of $X^{(s)}_{u}(x)$ by an amount (recall \eqref{crossover})
 $$\bbE[ (X_\gep-X_{s,\gep})(x)\bar X^{(s)}_u(x)]= K_{s+u,\gep,0}(x)-K_{s,\gep,0}(x) =:K^{(s)}_{u,\gep,0}.$$ 
Thus since $X^{(s)}_{u}(x)$ is Brownian motion  we obtain that 
\begin{multline}\label{bbbb}
  V^{(q,r)}_{s,\gep}
 =\sqrt{\frac{\pi t}{2}}\bP_{q+\sqrt{2d}s- \bar X_s-\rho_r(s)}\left[\forall u\in [0,t-s], B_u\ge \rho_{r+s}(u)-\sqrt{2d}(u-K^{(s)}_{u,\gep,0})\right].
 \end{multline}
 Setting $a:=q+\sqrt{2d}s- \bar X_s-\rho_r(s)$, using Lemma \ref{labesse} item $(i)$ (and the fact that $\mathfrak g_t(a)$ converges to $a$) we obtain that 
 \begin{equation}\label{woopz}
V^{(q,r)}_{s,\gep}=\sqrt{\frac{t}{t-s}}\mathfrak g_t(a)\bP_{t,a}\left[\forall u\in [0,t-s], B_u\ge \rho_{r+s}(u)-\sqrt{2d}(u-K^{(s)}_{u,\gep,0})\right].
 \end{equation}
We have $\lim_{\gep\to 0}\sqrt{\frac{t_{\gep}}{t_{\gep}-s}}\mathfrak g_{t_{\gep}-s}(a)=a$.
Furthermore \eqref{woopz} implies that $V^{(q,r)}_{s,\gep}$ is bounded in $L^5$ and hence tight in $L^4$.
Thus to conclude we need to show that the following convergence holds almost surely.
\begin{multline}
 \lim_{\gep\to 0} \bP_{t_{\gep},a}\left[\forall u\in [0,t_{\gep}-s], B_u\ge \rho_{r+s}(u)-\sqrt{2d}(u-K^{(s)}_{u,\gep,0})\right]\\
 =  \bQ_{a} \left[\forall u\in [0,t_{\gep}-s], B_u\ge \rho_{r+s}(u)\right].
\end{multline}
The proof is very similar to that of \eqref{browntobess}. Repeating the proof  of \eqref{steppun} (first line) displayed in the appendix, we obtain that 
\begin{multline}
  \lim_{\gep\to 0} \bP_{t_{\gep},a}\left[\forall u\in [0,t_{\gep}-s], B_u\ge \rho_{r+s}(u)-\sqrt{2d}(u-K^{(s)}_{u,\gep,0})\right]\\
 =  \lim_{T\to \infty} \lim_{\gep\to 0}\bQ_{a,t_{\gep}} \left[\forall u\in [0,T ], B_u\ge \rho_{r+s}(u)-\sqrt{2d}(u-K^{(s)}_{u,\gep,0}) \right].
\end{multline}
Now since $\sqrt{2d}(u-K^{(s)}_{u,\gep,0})$ converges uniformly to $0$ for any fixed $T$, using the convergence in total variation for 
$\bQ_{a,t_{\gep}} \left[ (B_u)_{u\in[0,T]}\right]$ mentioned in \eqref{dfljfdksl}, we obtain that 
\begin{multline}
  \lim_{T\to \infty} \lim_{\gep\to 0}\bQ_{a,t_{\gep}} \left[\forall u\in [0,T ], B_u\ge \rho_{r+s}(u)-\sqrt{2d}(u-K^{(s)}_{u,\gep,0}) \right]\\ =
  \bQ_a \left[\forall u\in [0,T ], B_u\ge \rho_{r+s}(u)\right]
\end{multline}
which concludes the proof.

\qed

\section{Degeneracy of critical GMC}\label{lastwalkz}

The aim of this section is to prove Proposition \ref{degeneracy} (a proof of Lemma \ref{topzzzz} is presented in Appendix \ref{toupz}).
The main input for the proof is the following variant of \cite[Proposition 2.4]{critix}.  Using the fact that $\gO$ is a smaller set, we obtain an upper bound for  $
\sup_{x\in \gO}\bar X_t(x)$ which is significantly smaller than the one that one would obtain on $[0,1]^d$.

\begin{proposition}\label{smalltek}
 For any $\alpha<1$ we have
\begin{equation}
 \sup_{x\in \gO} \left( \bar X_t(x)- \sqrt{2}t +\frac{\alpha \rho(t)}{\sqrt{2\log 2}} \right)<\infty
\end{equation}
\end{proposition}
\noindent The proof (which is a simpler version of that of \cite[Proposition 2.4]{critix}) is presented in Appendix \ref{smalltakk}. 
We set $\bar \rho(t)=\frac{1}{2}\rho(t)$.
Define 
$$\mathfrak  B_q:= \{\forall x\in \gO, \forall t\ge 0 \ :  \bar X_t(x)<  \sqrt{2}t - \bar \rho(t)+q \} $$
Proposition \ref{smalltek} implies that $\lim_{q\to \infty} P(\mathfrak B_q)=1$. On the other hand it is not difficult to check that the restriction of our different processes to $\mathfrak B_q$ converges to $0$ in $L^1$.

\begin{proposition}\label{bigaek}
We have 
\begin{equation}
\lim_{t\to \infty} \bbE[|D_t| \ind_{\mathfrak  B_q}]= \lim_{t\to \infty} \sqrt{t} \bbE[M_t \ind_{\mathfrak  B_q}]=
 \lim_{\gep\to 0} \sqrt{\log (1/\gep)} \bbE[M_{\gep} \ind_{\mathfrak  B_q}]=0.
\end{equation}

\end{proposition}

\begin{proof}[Proof of Proposition \ref{degeneracy}]
Let us focus on the case of $D_t$ since the others are identical.
For every $q$,  $\lim_{t\to \infty} D_t \ind_{\mathfrak  B_q}=0$ in probability, and we conclude that $D_t\ind_{\bigcup_{q\ge 1}\mathfrak B_q}$  (which almost surely coincides with $D$) also converges in probability to zero.

\end{proof}

\begin{proof}[Proof of Proposition \ref{bigaek}]
Let us set for this proof 
$$ B_{q,t}(x):=\{\forall s\in[0,t],\bar X_t(x)\le \sqrt{2}t+q- \bar \rho(t) \}.$$
Note that we have
\begin{equation}
\bbE[M_t \ind_{\mathfrak  B_q}]\le \int_\gO \bbE\left[ e^{\sqrt{2}X_t(x)-K_t(x)} \ind_{B_{q,t}(x)}\right]\mu(\dd x)
\end{equation}
The integrand above does not depend on $x$.
Using Cameron-Martin formula and Lemma \ref{labesse} (item $(i)$) we have 
\begin{equation*}
\sqrt{\pi t/2}\bbE\left[ e^{\sqrt{2}X_t(x)-K_t(x)} \ind_{B_{q,t}(x)}\right]\le q \bQ_{q,t}\left[ \forall s\in [0,t], B_s\ge \bar \rho(s) \right].
\end{equation*}
Using stochastic comparision between $\bQ_{q,t}$ and $\bQ_{q}$ we have
\begin{equation}
  \bQ_{q,t}\left[ \forall s\in [0,t], B_s\ge \bar \rho(s) \right]\le   \bQ_{q} \left[ \forall s\in [0,t], B_s\ge \bar \rho(s) \right].
\end{equation}
Since $\bar\rho$ does not satisfies \eqref{besselcond}, the r.h.s.\ converges to zero and thus 
\begin{equation}
\lim_{t\to \infty}\sqrt{\pi t/2} \bbE[M_t \ind_{\mathfrak  B_q}]=0. 
\end{equation}
The same computation works for  $M_{\gep}$. In the case of  $D_t$, we have
\begin{equation}
 \bbE[|D_t| \ind_{\mathfrak  B_q}]= \int_\gO\bbE\left[ |\sqrt{2}K_t(x)- X_t(x)| 
 e^{\sqrt{2}X_t(x)-K_t(x)} \ind_{B_{q,t}(x)}\right]\mu(\dd x)
\end{equation}
We split the term in the expectation in the following manner
$$|\sqrt{2}K_t(x)- X_t(x)|\le  (\sqrt{2}t+q- \bar X_t(x))+ |\sqrt{2}K_0(x)- X_0(x)-q|.$$
We have (from Cameron-Martin formula
\begin{equation}
 \bbE\left[   (\sqrt{2}t+q- \bar X_t(x))
 e^{\sqrt{2}X_t(x)-K_t(x)} \ind_{B_{q,t}(x)}\right]= q \bQ_q \left[ \forall s\in [0,t], B_s\ge \bar \rho(s) \right],
\end{equation}
which tends to zero when $t\to \infty$. For the other term, we also use Cameron Martin formula. We obtain that 
\begin{multline}
  \bbE\left[   |\sqrt{2}K_0(t)- X_0(x)-q|
 e^{\sqrt{2}X_0(x)-K_t(x)} \ind_{B_{q,t}(x)}\right]\\
 =\bbE\left[ |X_0(x)+q| \right]\bQ_{q,t} \left[ \forall s\in [0,t], B_s\ge \bar \rho(s) \right].
\end{multline}
The first term in the r.h.s.\ is bounded uniformly in $\gO$ and the second one has been shown to be $o(t^{-1/2})$ (and thus in particular tends to zero).

\end{proof}

\section{Extension of our results}\label{sketch}

In this section, we sketch the proof of the claims made in Section \ref{extense}. 
In Section \ref{relax} we justify the claim that the results also hold when \eqref{clazomp} is replaced  by the weaker assumption \eqref{root}.
In Sections \ref{convj} and \ref{devj} we give an outline for the proof of Theorem \ref{outro}

\subsection{Relaxing the assumption on $\mu$}\label{relax}

Let us briefly outline the argument establishing the convergence  of $D_t(\cdot,\mu)$ and that of  $\sqrt{\frac{\pi t}{2}}M_t(\cdot,\mu)$. The same ideas can be used for $\sqrt{\frac{\pi \log (1/\gep)}{2}}M_\gep(\cdot,\mu)$. 
Since $\mu_n$ satisfies \eqref{clazomp}, using  Proposition \eqref{uint}, for for any fixed $n\ge 1$, $E\in \mathcal B_b$  and $q\ge 1$, the martingale $D^{(q)}_t(E,\mu_n)$ is uniformly integrable. 
We observe that 
\begin{equation}\label{apirox}
 \bbE\left[ |D^{(q)}_t(E,\mu)- D^{(q)}_t(E,\mu_n)| \right]=q (\mu-\mu_n)(E)= q\mu(E\cap A^{\cc}_N).
\end{equation}
In other words $D^{(q)}_t(E,\mu)$ is well approximated uniformly in $t$ by a uniformly integrable martingale.
This is the right-setup to apply Proposition \ref{metaprop}
($t$ is playing the role of $n$ and $n$ that of $r$) and from which we deduce that $D^{(q)}_t(E,\mu)$ is uniformly integrable and hence converge.  
From \eqref{apirox} and Fatou we also obtain the $L^1$ convergence for the limits
\begin{equation}\label{sdfgh}
 \lim_{n\to \infty} D^{(q)}_{\infty}(E,\mu_n)= D^{(q)}_{\infty}(E,\mu).
 \end{equation}
 The convergence of $D_t$ as a measure follows from the usual argument bases on Proposition \ref{weako}.
For $M_t$ we know from Proposition \ref{dor} that for every $n\ge 1$ we have the following convergence in $L^1$ 
\begin{equation}\label{sdfghj} 
 \lim_{t\to \infty} M^{(q)}_t(E,\mu_n)= D^{(q)}_{\infty}(E,\mu_n)
\end{equation}
On the other hand we have 
\begin{equation}
  \bbE\left[ |M^{(q)}_t(E,\mu)- M^{(q)}_t(E,\mu_n)| \right]= 
  \mu(E\cap A^{\cc}_N)\bP_q\left[ \forall s\in [0,t], B_s \ge 0\right]
\end{equation}
This implies  (cf. Lemma \ref{labesse} item (i)) that   $$\lim_{n\to \infty}\sup_{t\ge 1}\sqrt{\frac{\pi t}{2}} \bbE\left[ |M^{(q)}_t(E,\mu)- M^{(q)}_t(E,\mu_n)| \right]= 0$$
 which is exactly what is required to interchange limits in $t$ and $n$. Hence we deduce from \eqref{sdfgh}-\eqref{sdfghj} that we have the following convergence in $L^1$. 
\begin{equation}
\lim_{t\to \infty}\sqrt{\frac{\pi t}{2}}M^{(q)}_t(E,\mu)=\lim_{n\to \infty}  
 \lim_{t\to \infty} \sqrt{\frac{\pi t}{2}}M^{(q)}_t(E,\mu_n)=D^{(q)}_{\infty}(E,\mu). 
\end{equation}
Again the convergence of $M^{(q)}_t(E,\mu)$ as a measure follows.

\subsection{Proof of Theorem \ref{outro}: Convergence}\label{convj}
We assume that $\rho$ satisfies \eqref{besselcond}.
The field $Y_t$ has stronger correlations than $X_t$ therefore we have 
\begin{equation}\label{toctic}
 \sup_{x\in \bbR^d} \sup_{t\ge 0}\left(\sqrt{2d}t- Y_t(x)\right)<\infty.
\end{equation}
The proof of \eqref{toctic} can be obtained by replicating that of \cite[Proposition 2.4]{critix}.
Using this we can define reduce the proof of convergence of $F_t$ to that of $F^{(q)}_t$ which is defined via the same truncation proceedure as \eqref{trunk1}. For the second truncation in $r$ we slightly modify the definition in \eqref{trunk2} adding a factor $\sqrt{2d}$ in front of $\rho_r$   (defined by $\rho_r(u)=\rho(r+u)+\rho(r)$).
\begin{equation}\begin{split}
\bar A^{(q,r)}_t(x)&:= \{ \forall s\in [0,t], \ : \ Y(s)\le \sqrt{2d}(s-\rho_r(s))+q  \}, \\
N^{(q,r)}_t(E)&:=\int_E e^{\sqrt{2d}Y_t(x)-dt}\ind_{A^{(q,r)}_t(x)}\dd x=:\int_{E} \bar W^{(q,r)}_t\dd x ,\\
F^{(q,r)}_t(E)&:=\int_E (\sqrt{2d}(t- \rho_r(t))+q -Y_t(x))\bar W^{(q,r)}_t\dd x=:\int_{E} \bar Z^{(q,r)}_t \dd x.
\end{split}\end{equation}
With these definitions, the proof of the convergence of $\sqrt{ 2t/\pi} N_t$ and $F_t$ follows exactly the same plan as that of 
$\sqrt{ 2t/\pi} M_t$ and $D_t$. 
The only step which requires a (minor) modification is the computation of the second moment of  $N^{(q,r)}_t$. 
Setting 
$$\bar u(t,x,y)=  \left[\log \left(\frac{1}{|x-y|}\right)+\rho \left( \log \left(\frac{1}{|x-y|}\right) \right)\right]\wedge t$$
Adapting the proof of \eqref{tek1}  in a straightforward manner 
we show that
 \begin{equation}\label{teknn}
  \  \bbE\left[ \bar W^{(q,r)}_t(x) \bar W^{(q,r)}_t(y)\right] \le C e^{d \bar u- 2d\rho(\bar u)} (\bar u+1)^{-3/2}  ( t-\bar u+1)^{-1}.
 \end{equation}
 Now since 
 $$ d \bar u- 2d\rho(\bar u)\le d \log \left(\frac{1}{|x-y|}\right)-d\rho  \left( \log\left(\frac{1}{|x-y|}\right)\right)$$
 it is not difficult to check that \eqref{teknn} implies that 
 \begin{equation}
    \sup_{t\ge 0}  \bbE\left[ t N^{(q,r)}_t(E)^2 \right]\le \gl(E) e^{-d \rho(\log (1/ \Diam(E))} \log (\Diam (E))^{-1/2}.
 \end{equation}
The above estimate is sufficient to show that the measure $F_{\infty}$ obtained in the limit is atomless.

\subsection{Proof of Theorem \ref{outro}: Degeneracy}\label{devj}
Let us now assume that $\rho$ fails to satisfy \eqref{besselcond}.
Going beyond \eqref{toctic} when adapting the proof of \cite[Proposition 2.4]{critix}, it is possible to prove that
\begin{equation}\label{toctic2}
 \sup_{x\in \bbR^d} \sup_{t\ge 0}\left(\sqrt{2d}(t- \rho(t))- Y_t(x)\right)<\infty.
\end{equation}
This implies that in particular that given $R$ there exists a random $q_0$ such that for all $E\subset B(0,R)$ $t\ge 0$  and $q\ge q_0$
$$ N^{(q,0)}_t(E)= N_t(E)  \text{ and } F^{(q,0)}_t(E)= F_t(E)+ q N_t(E)\ge F_t(E) $$
To conclude we simply need to  prove that 
\begin{equation}
 \lim_{t\to \infty}\sqrt{t}\bbE[N^{(q,0)}_t]=\lim_{t\to \infty}\bbE[F^{(q,0)}_t]=0,
\end{equation}
which can be done by replicating the proof of Proposition \ref{degeneracy} from  Section \ref{lastwalkz}.

 \appendix

 \section{Proof of technical estimates}\label{teklesproofs}

 \subsection{Proof of Lemma \ref{labesse}}\label{laabesse}
  The first two estimates are direct consequences of the reflection principle, the third one is obtained by applying  the optional stopping Theorem to the martingale $e^{2aB_t-2a^2 t}$.
  To prove $(iv)$, since $\rho_r\le \rho_0$ we have
\begin{equation}\label{took}
\bQ_a\left[ \exists s\ge 0, \  \beta_s\le \rho_r(s)\right]
\le \bQ_a\left[ \exists s\in [0,T], \  \beta_s\le \rho_r(s)\right] + 
\bQ_a\left[ \exists s\ge T, \  \beta_s\le \rho_0(s)\right]
\end{equation}
Using  \eqref{dvorerdos} the second term in the r.h.s.\ of \eqref{took} tends to zero when $T\to \infty$.
To conclude it is sufficient to observe that since $\lim_{r\to \infty} \sup_{[0,T]}\rho_r(s)=0$, we have for every $T>0$
\begin{equation}
 \lim_{r\to \infty} \bQ_a\left[ \exists s\in [0,T], \  \beta_s\le \rho_r(s)\right]=0.
\end{equation}
To prove $(v)$ our starting point we apply Doob-McKean  cf.\ \eqref{dmk} to obtain that for any event  $A_0$
 \begin{equation}\label{toopz}
 \bQ_{a,t}\left[  (B_s)_{s\in[0,t]}\in  A_0 \right]=\frac{\bQ_a\left[ \frac{a}{\beta_t} \ind_{\{(\beta_s)_{s\in [0,t]}\in  A_0\}}\right]}{{\bP_a\left( \forall s\in [0,t], B_s\ge 0\right)}}.
 \end{equation}
 Applying the above to $A_0=A\cap\left\{B_t> \delta \sqrt{\pi t} \right\}$
we obtain that for any $\delta>0$.
\begin{equation}
 \bQ_{a,t}\left[  (B_s)_{s\in[0,t]}\in A \right]\le \frac{a}{\delta\sqrt{\pi t}}\frac{\bQ_{a}\left[  (\beta_s)_{s\in [0,t]}\in A \right]}{\bP_a\left( \forall s\in [0,t], B_s\ge 0\right)}+\bQ_{a,t}\left(B_t\le \delta \sqrt{\pi t}  \right)
\end{equation}
Now item $(i)$ implies that $\bP_a\left( \forall s\in [0,t], B_s\ge 0\right)\ge a (\pi t)^{-1/2}$ for $t\ge t_0(a)$ sufficiently large.
Furthermore FKG inequality applied to the measure $\bP_{t}[ \ \cdot  \ | \ B_t\ge 0]$ (see \cite{preston}) implies that 
\begin{equation}
\bQ_{a,t}\left(B_t\le \delta \sqrt{\pi t}  \right)\le \bP_a[ B_t  \le \delta \sqrt{\pi t} \ | \  B_t\ge 0]\le 2 \bP_a[ B_t  \in [0, \delta \sqrt{\pi t} ]]\le \sqrt{2} \delta.
\end{equation}
Taking $\delta= \sqrt{\bQ_{a}\left[  (\beta_s)_{s\in [0,t]}\in A \right]}$, we obtain the desired conclusion. 
To prove (vi) we are going to prove the two following identities
\begin{equation}\begin{split}\label{steppun}
  \lim_{T\to\infty}\limsup_{t\to \infty} \big|  \bQ_{a,t}\left[ \forall s\in [0,t], B_s\ge \rho_r(s) \right]&-   \bQ_{a,t}\left[ \forall s\in [0,T], B_s\ge \rho_r(s) \right]\big|=0,\\
 \lim_{T\to \infty}\lim_{t\to \infty} \bQ_{a,t}\left[ \forall s\in [0,T], B_s\ge \rho(s) \right] & =\bQ_{a}\left[ \forall s\ge 0, \beta_s\ge \rho(s) \right]
 \end{split}
\end{equation}
To show that the first line in \eqref{steppun} holds
we use the fact that  
  \begin{multline}
  \left|  \bQ_{a,t}\left[ \forall s\in [0,t], B_s\ge \rho_r(s) \right]-   \bQ_{a,t}\left[ \forall s\in [0,T], B_s\ge \rho_r(s) \right]\right| \\
   \le     \bQ_{a,t}\left[ \exists s\in [T,t], B_s\le \rho_r(s) \right].
   \end{multline}
From $(v)$ we have
   \begin{equation}
     \bQ_{a,t}\left[ \exists s\in [T,t], B_s\le \rho(s) \right]\le (1+\sqrt{2})  \sqrt{\bQ_{a}\left[  \exists s\ge T , \beta_s\le \rho(s) \right]}.
   \end{equation}
and according to \eqref{dvorerdos}, the r.h.s.\ tends to zero when $T\to \infty$, which concludes the proof of \eqref{browntobess}. 
The convergence in the second line  in \eqref{steppun} is a direct consequence of the following convergence in total variation
\begin{equation}\label{dfljfdksl}
\lim_{t\to \infty} \bQ_{a,t}\left[ (B_s)_{s\in[0,T]}\in \cdot  \right]  = \bQ_a\left[ (\beta_s)_{s\in[0,T]}\in \cdot\right],
   \end{equation}
(that can be inferred from Doob-McKean identity \eqref{dmk}).
 \qed
 
 \subsection{Proof of Proposition \ref{listov}}\label{liistov}

 For $(i)$, since all the cases are similar we only provide a detailed proof for $D^{(q,r)}_t$
 (note that the martingale (resp. supermartingale) property is proved for $M^{(q)}_t$ res $D^{(q)}_t$ in \cite[Lemma 2.3]{critix}).
 It is sufficient to check that $Z^{(q,r)}_t(x)$ is a supermartingale and then integrate over $E$.
The process $Z^{(q,r)}_t(x)$  is continuous and
 Itô calculus yields 
 $$\dd  Z^{(q,r)}_t(x)=  \left(\sqrt{2d} Z^{(q,r)}_t(x)-  W^{(q,r)}_t(x)\right) \dd \bar X_t(x)- \rho'_r(t) W^{(q,r)}_t(x)\dd t.$$ 
We can then conclude from the fact that $\rho'_r(t)\ge 0$ and $\bar X_\cdot(x)$ is a standard Brownian motion.
For $(ii)$, to prove the first line in \eqref{rlimit}, it is sufficient to show that 
\begin{equation}
\lim_{r\to \infty}\liminf_{t\to \infty}\bbE\left[ Z^{(q,r)}_t(x)\right]=q
\end{equation} 
 \noindent Furthermore we have
 \begin{equation}\label{dope}\begin{split}
 \bbE[Z^{(q,r)}_t(x)]&= \bbE\left[ (\sqrt{2d} t+q-\rho_r(t)-\bar X_t(x))e^{\sqrt{2d}X_t-d\bbE[X^2_t]} \ind_{A^{(q,r)}_t(x)}\right]\\
 &=  \bbE\left[ (q-\rho_r(t)-\bar X_t(x))\ind_{\{ \forall s\in [0,t], \bar X_s\le q-\rho_r(s) \} }\right]\\&=
 q \bQ_q\left[\forall s\in [0,t],\ \beta_s\ge \rho_r(s) \right]- \rho_r(t)\bP_q[\forall  s\in [0,t], B_s\ge \rho_r(s)].
 \end{split}
 \end{equation}
Let us take a look at the second term.  
Since \eqref{besselcond} implies that $\rho(u)=o(u^{-1/2})$ we have
\begin{equation}
\lim_{t\to \infty}\rho(t)\bP_q[\forall  s\in [0,t], B_s\ge 0]=0.
\end{equation}
Concerning the first term, we have 
\begin{equation}
\bQ_q\left[\forall s\in [0,t],\ \beta_s\ge \rho_r(s)\right]\ge  \bP_q\left[\forall s>0,\ \beta_s\ge \rho_r(s) \right],
\end{equation}
and Lemma \ref{labesse} allows us to conclude.
For the second line in \eqref{rlimit},
we obtain, applying Cameron-Martin formula like in \eqref{dope} 
\begin{equation}\begin{split}\label{sssds}
\bbE[W^{(q)}_t(x)-W^{(q,r)}_t(x)]&= \bP_q[\forall s\in[0,t], B_s\ge 0 \ ; \ \exists u\in [0,t], 
B_u\le \rho_r(u)]\\ &= \bP_q[\forall s\in[0,t], B_s\ge 0] \bQ_{q,t}\left[  \exists u\in [0,t], B_u\le \rho_r(u) \right].¨
\end{split}
\end{equation}
Now we have from Lemma \ref{labesse} item $(i)$
\begin{equation}\label{straight}
 \bP_q[\forall s\in[0,t], B_s\ge 0]\le q \sqrt{\frac{2}{\pi t}}.
\end{equation}
Using the stochastic comparison between $\bQ_q$ and $\bQ_{q,t}$ which follows from FKG, we have for every $t\ge 0$
\begin{multline}
 \bQ_{q,t}\left[  \exists u\in [0,t], B_u\le \rho_r(u) \right] \\ \le  \bQ_q\left[  \exists u\in [0,t], \beta_u\le \rho_r(u)\right]\le \bQ_q\left[  \exists u\ge 0, \beta_u\le \rho_r(u)\right]
\end{multline}
and the r.h.s.\ tends to zero when $r\to \infty$ by Lemma \ref{labesse}. 
For the third line in \eqref{rlimit}, we observe that 
\begin{multline}\label{triphop}
\bbE[W^{(q)}_\gep(x)-W^{(q,r)}_{\gep}(x)]\\
= \bP_q[\forall s\in[0,t_{\gep}], B_s\ge -\delta(s,\gep) \ ; \ \exists u\in [0,t_{\gep}],\
B_u\le\rho_r(u)-\delta(u,\gep)].
\end{multline}
where 
$$\delta(s,\gep):=\sqrt{2d}(s-\overline K_{s,\gep}(x))=\int^s_0\sqrt{2d} (1-Q_{u,\gep}(x))\dd u=:\int^s_0 \eta(u,\gep)\dd u.$$
The l.h.s.\ of \eqref{triphop} can be bounded in the same way as \eqref{sssds}, but the $\delta(u,\gep)$ in the formula  requires a bit of technical care.
We set $\delta_{\gep}=\delta(t_{\gep}/2,\gep)$, we use the bound  $\delta(s,\gep)\le \delta_{\gep}$ for $t\le t_{\gep}/2$ and  $\delta(s,\gep)\le \delta_{\gep}\le C$ (for some finite constant $C$) for $t\in [t_{\gep}/2,t_{\gep}]$ we obtain
\begin{multline}
 \bbE[W^{(q)}_\gep(x)-W^{(q,r)}_{\gep}(x)]\\
 \le  \bP_q[\forall s\in[0,t_{\gep}], B_s\ge -C \ ; \ \exists u\in [t_{\gep/2},t_{\gep}],\
B_u\le\rho_r(u)]\\
+ \bP_q[\forall s\in[0,t_{\gep}/2], B_s\ge -\delta_{\gep} \ ; \ \exists u\in [0,t_{\gep/2}],\
B_u\le\rho_r(u)].
\end{multline}
Since we have 
\begin{equation}
 \max(\bP_q[\forall s\in[0,t_{\gep}], B_s\ge -C ], \bP_q[\forall s\in[0,t_{\gep}/2], B_s\ge -\delta_{\gep}])\le C'(\log (1/\gep))^{-1/2}.
\end{equation}
we can conclude the proof of \eqref{rlimit} if we show that that the conditional probabilities converge to zero, or more precisely 
\begin{equation}\begin{split}     \label{exhaust}            
 \lim_{\gep\to 0}\bQ_{q+C,t_{\gep}}[ \ \exists u\in [t_{\gep/2},t_{\gep}],\
B_u\le\rho_r(u)+C]&=0,\\
           \lim_{r\to \infty}\limsup_{\gep \to 0}\bQ_{q+\delta_{\gep},t_{\gep}/2}[ \exists u\in [0,t_{\gep/2}],\
B_u\le\rho_r(u)+\delta_{\gep}]&=0
     \end{split}
\end{equation}
For the first line, we  observe that since $\rho_r$ is a lower-envelope
\begin{equation}
  \lim_{\gep\to 0}\bQ_{q+C}[ \ \exists u\in [t_{\gep/2},t_{\gep}],\
B_u\le\rho_r(u)+C]=0,
\end{equation}
and use Lemma \ref{labesse} item (v) to get that the statement also holds when $\bQ_{q+C}$ is replaced by  $\bQ_{q+C,t_{\gep}}$. For the second line in \eqref{exhaust} we can repeat the proof of \eqref{browntobess} with minor modification to obtain that 
\begin{equation}
 \lim_{\gep \to 0}\bQ_{q+\delta_{\gep},t_{\gep}/2}[ \exists u\in [0,t_{\gep/2}],\
B_u\le\rho_r(u)+\delta_{\gep}]=\bQ_{q}[ \exists u\ge 0,\
B_u\le\rho_r(u)].
\end{equation}
and conclude from Lemma \eqref{labesse} item (iv). Let us finally briefly discuss the three remaining items.
For $(iii)$ we refer to \cite[Corollary 2.5]{critix} and for  $(iv)$  to \cite[Proposition 2.4]{critix}.
Finally $(v)$ is direct consequence of the combination of $(iii)$ and $(iv)$.
\qed

 \subsection{Proof of Proposition \ref{smalltek}}\label{smalltakk}

\noindent
For $n\ge 0$, we split $\gO$ into $2^n$ subsets of diameter $d_n:=2^{-n}\prod^n_{i=1}(1-a_n)=2^{-n} e^{-\vartheta(n)} $.
We set 
$$I_{n,1}:= \gO\cap\left[ 0,  d_n \right]$$
and let $I_{n,i}$ ($i\in \lint 1,2^n\rint$) denote the $2^{n}$ disjoint isometric copies of $I_{n,1}$ that can be found in $\gO$ (which we order from left to right).
We let $a_i$ denote the minimal element of $I_{n,i}$.
We set $t_n:=\log (1/\theta_n)=(\log 2) n +\vartheta(n)$.
We define
\begin{equation}
A_{n,i}:= [t_n, t_{n+1})\times I_{n,i}, \quad 
 Y_{n,i}=\bar X_{t_n}(a_i) , \quad 
Z_{n,i}:= \max_{(t,x)\in A_{n,i}} \left(\bar X_{t}(x)-Y_{n,{i}}\right). \end{equation}
The statement we want to prove can be restated as follows, for every $\alpha\in(0,1)$  we have
\begin{equation}\label{dabest}
 \sup_{n\ge 0} \max_{{\bf i}\in I_n} \left[ Y_{n,{ i}}+Z_{n,{i}} - \sqrt{2} t_n +\frac{\alpha \vartheta(n)}{\sqrt{2}}\right]<\infty. 
 \end{equation}
The following result  ensures that the tail of $Z_{n,{\bf i}}$ are uniformly sub-Gaussian. 
The statement is almost identical to that of \cite[Lemma A.1]{critix} (which we provide as a reference for the proof).
\begin{lemma}\label{ggttail}
 There exists a constant $c>0$ such that for all $n\ge 0$, ${\bf i}\in \mathbb Z^d$ and $\gl\ge 0$
 \begin{equation}\label{kilaro}
 \bbP[Z_{n,i}\ge \gl ] \le 2 \exp\left( - c \gl^2  \right).
\end{equation}
\end{lemma}
\noindent Combining \eqref{kilaro} and the fact that $Y_{n,i}$ is a centered Gaussian with variance $t_n$, we obtain
\begin{equation}\label{convovo}
 \bbP\left[ Y_{n,i}+Z_{n, i} \ge A \right] \le \frac{2}{\sqrt{2\pi t_n}}\int_{\bbR}  \exp\left( - \frac{u^2}{2 t_n} - c (A-u)_+^2  \right) \dd u.
\end{equation}
This implies  that 
\begin{equation*}
\sum_{n\ge 0}  \sum_{i =1}^{2^n}\bbP \left[ Y_{n, i}+Z_{n, i} - \sqrt{2d}t_n  +\frac{\alpha \vartheta(n)}{\sqrt{2}}\ge  0\right]\le  C \sum_{n\ge 0}  
\sum_{i \in I_n}  e^{\alpha \vartheta(n)- t_n }= C \sum_{n\ge 0}  e^{(\alpha-1) \vartheta(n)}.
\end{equation*}
The last sum is finite ($\alpha<1$ and $\vartheta(n)>n^{1/3}$ for large $n$) so we can conclude  using Borel-Cantelli.
\qed

  \subsection{Proof of Lemma \ref{topzzzz}} \label{toupz}
  
  Using the notation of the the previous proof, 
  we  let   $d^-_n= a_{n+1}    d_n$. 
Note that $d^-_{n-1}$ and $d_{n-1}$  are respectively  the minimal  and maximal distances between  an element of  $I_{n,1}$ and one of $I_{n,2}$ for $n\ge 1$.
Using the facts that the diagonal has zero measure under $\mu^{\otimes 2}$ and that 
  $I_{n,2i-1}\times I_{n,2i}$ is the image of $I_{n,1}\times I_{n,2}$ by a measure preserving isometry we obtain that
  \begin{equation}\begin{split}
  \int_{\Omega\times \gO} \frac{\mu(\dd x)\mu(\dd y)}{|x-y| e^{\alpha \vartheta(|x-y|)}}
  &=\sum_{n\ge 1} \sum_{i=1}^{2^{n-1}}\int _{(I_{n,2i-1}\times I_{n,2i})\cup(I_{n,2i-1}\times I_{n,2i})}  \frac{\mu(\dd x)\mu(\dd y)}{|x-y| e^{\alpha \vartheta(|x-y|)}}\\
  &= \sum_{n\ge 1}  2^n \int_{I_{n,1}\times I_{n,2}}   \frac{\mu(\dd x)\mu(\dd y)}{|x-y| e^{\alpha \vartheta(|x-y|)}}\\
  &\le \sum_{n\ge 1} 2^{n} \mu^{\otimes 2}\left(I_{n,1}\times I_{n,2}\right)  (d^-_{n-1})^{-1} e^{-\alpha \vartheta(|\log d_{n-1}|)}\\
  &=  2\sum_{n\ge 0} a^{-1}_n e^{\vartheta(n)}e^{-\alpha \vartheta(|\log d_n|)}.
  \end{split}\end{equation}
Since  $\vartheta$ is regularly varying and $|\log(d_n)|= n \log 2 +o(n)$  we have $\vartheta(|\log d_n|)=\sqrt{\log 2}\vartheta(n)(1+o(1))$,
using the fact that $\vartheta(n)\ge n^{1/3}$ and $a^{-1}_n\le n^{2/3}$ for $n$ sufficiently large, this allows to conclude \qed.

  \subsection{Proof of Proposition \ref{lateknik}}

We first prove \eqref{tek1} and then indicate how \eqref{tek2} follows from the same computation with minor modifications.
 Using Cameron-Martin formula  we obtain that 
 \begin{multline}\label{a11}
 \!\! \bbE\left[W^{(q,r)}_t(x) W^{(q,r)}_t(y)\right]\\ =
  e^{2d K_t(x,y)} \bbP\left[ \forall s\in [0,t],  \bar X_s(x) \! \vee  \! \bar X_s(y)\le  q - \sqrt{2d} \bar K_s(x,y)-\rho_r(s) \right].
 \end{multline}
To estimate the probability on the right-hand side, we perform a couple of simplification. 
Recalling the definition of $u(x,y,t)$, from \eqref{thelogbound}, there exists  $C>0$ such that 
$$   \bar K_s(x,y) \ge s\wedge u - C.$$
For the remainder of the proof, we set $q'=q+C\sqrt{2d}$. Another observation is that $\bar X_{\cdot}(x)$ and $\bar X_{\cdot}(y)$ play symmetric roles so that 
adding the restriction $\bar X_u(x)\le \bar X_u(y)$ only changes the probability of the event  in the r.h.s.\ of \eqref{a11} by a factor $1/2$. It is thus smaller than 
\begin{multline}\label{wopwopwop}
2 \bbP\left[ \forall s\in [0,t], \ \bar X_s(x) \! \vee  \! \bar X_s(y)\le  q' - \sqrt{2d}(u\wedge s)-\rho_r(s) \ ; \ \bar X_u(x)\le \bar X_u(y) \right] 
 \\ \le 2 \bbP\left[ F_1 \cap F_2 \cap F_3\right]= 2 \bbE\left[ \ind_{F_1} \bbP[ F_2\cap F_3  \ | \ \cF_u] \right].
 \end{multline}
 where the events $(F_i)^3_{i=1}$ are defined by
 \begin{equation}
  \begin{split}
   F_1&:=  \{ \forall s\in [0,u],  \bar X_s(x)\le  q' - \sqrt{2d} s -\rho_r(s) \},\\
   F_2&:=\{ \forall s\in [0,t-u],  X^{(u)}_s(x)\le  q' - \bar X_u(x)-\sqrt{2d}u-\rho_r(s) \},\\
   F_3&:=\{ \forall s\in [0,t-u],  X^{(u)}_s(y)\le  q' - \bar X_u(x)-\sqrt{2d}u-\rho_r(s) \}.
  \end{split}
 \end{equation}
The inequality in \eqref{wopwopwop} is obtained by ignoring the requirement $\bar X_s(y)\le q'-\sqrt{2d}s-\rho_r(s)$ for $s\in [0,u]$, 
and using the fact that for $s\in[u,t]$
$$\{ \bar X_s(y)\le q'-\sqrt{2d} u-\rho_r(s)\} \cap\{  \bar X_u(x)\le \bar X_u(y) \}  \subset \{ X^{(u)}_s(y)\le  q' - \bar X_u(x)-\sqrt{2d}u-\rho_r(s)\}.$$ 
Since $X^{(u)}_{\cdot}(x)$ and $X^{(u)}_{\cdot}(y)$ are independent and independent of $\mathcal F_u$, the events $F_2$ and $F_3$ are conditionally independent. Hence we have 
\begin{equation}
\label{cronde1}
  \bbP\left[ F_2\cap F_3 \ | \ \cF_u\right] \le 1\wedge \left(\frac{2(q'-\bar X_u(x)-\sqrt{2d}u-\rho_r(u))^2_+}{\pi(t-u)}\right).
\end{equation}
Now to compute $\bbE\left[ \ind_{F_1} \bbP[ F_2\cap F_3  \ | \ \cF_u] \right]$, we use
 \eqref{bridjoux} and obtain that 
\begin{equation}\label{cronde2}
 \bbE[F_1 \ | \ \bar X_u(x)]\le 1\wedge \left(\frac{2q'(q'-\bar X_u(x)-\sqrt{2}u-\rho_r(u))_+}{u}\right).
\end{equation}
Putting together \eqref{cronde1} and \eqref{cronde2} and integrating w.r.t.\  $z=(q'-\bar X_u(x)-\sqrt{2}u-\rho_r(u))$ we have
\begin{multline}
\bbP\left[ F_1 \cap F_2 \cap F_3\right]\le  \int^{\infty}_0 \frac{e^{-\frac{(z-q'+\sqrt{2d} u+\rho_r(u))^2}{2u}}}{\sqrt{2\pi u}}
 \left(1\wedge \frac{2q'z}{u}  \right)     \left(1\wedge \frac{z^2}{\pi(t-u)}\right) \dd z\\
 \le \frac{C e^{-du-\sqrt{2d}\rho_r(u)}}{(u + 1)^{3/2}[(t-u)+ 1]}.
\end{multline}
To obtain the last inequality, we can replace $1\wedge \frac{2q'z}{u} $ by $1$ if $u\le 1$ and  by $\frac{2q'z}{u}$ if $u\ge 1$, and 
 $1\wedge \frac{z^2}{\pi(t-u)}$ by $1$ if $t-u\le 1$ or $\frac{z^2}{\pi(t-u)}$ if $t-u\ge 1$. 
 Since in \eqref{a11}  we have $e^{2d K_t(x,y)}\le C e^{2du}$ (cf. \eqref{thelogbound}) this concludes the proof of \eqref{tek1}.
For \eqref{tek2} we have
\begin{multline}
  \bbE\left[ W^{(q,r)}_\gep(x) W^{(q,r)}_\gep(y) \right] \\ \le
  e^{2d K_\gep(x,y)} \bbP\left[ \forall s\in [0,t],  \bar X_s(x)  \vee  \bar X_s(y)\le  q - \sqrt{2d} \bar K_{s,\gep,0}(x,y)-\rho_r(s) \right].
 \end{multline}
Recalling the definition of $v(x,y,\gep)$, we have  $K_\gep(x,y)\le v+C$
and 
$\bar K_{s,\gep,0}(x,y) \ge s\wedge v-C.$
We then repeat the proof of \eqref{tek1} to obtain the same conclusion with $u$ replaced by $v$. To finish we can restrict ourself to the proof of the first consequence $(i)$ since the proof of $(ii)$ is similar. We have 
\begin{equation}\begin{split}
 \bbE\left[ (M^{(q,r)}_t(E))^2 \right]&\le C \int_{E\times E} \frac{ e^{d u-\sqrt{2d}\rho(u)}\mu(\dd x) \mu(\dd y)}{ (u+1)^{3/2}  (t-u+1)}
 \\ 
 &\le  C t^{-1}\int_{E\times E} e^{d u-\sqrt{2d}\rho(u)}\mu(\dd x) \mu(\dd y)\\
 &\le C t^{-1}\int_{E\times E }\frac{\mu(\dd x) \mu(\dd y)}{|x-y|^d e^{\sqrt{2d}\rho(\log \frac1{|x-y|})}}.
\end{split}\end{equation}
In the last line we have used $\log \frac{1}{|x-y|}\ge u(x,y,t)$, hence the integral above is larger than 
\begin{equation*}\begin{split}
\int_{E\times E } \frac{\mu(\dd x) \mu(\dd y)}{|x-y|^d e^{\sqrt{2d}\log (\frac 1{|x-y|} )}}
&\le    e^{-(\sqrt{2}-1)\rho(1/\Diam(|E|)) }\int_{E\times E } \frac{\mu(\dd x) \mu(\dd y)}{|x-y|^d e^{\rho(\log \frac 1{|x-y|})}}\\
&\le    e^{-(\sqrt{2}-1)\rho(1/\Diam(|E|)) }\int_{E\times B(0,R)} \frac{\mu(\dd x) \mu(\dd y)\ind_{\{|x-y|\le 1\}}}{|x-y|^d e^{\rho(\log \frac 1{|x-y|})}}.
\end{split}\end{equation*}
This concludes the proof, defining $\mu_R$ by setting 
$$\frac{\dd \nu_R}{\dd \mu}(x) = C\int_{B(0,R)} \frac{\ind_{\{|x-y|\le 1\}}\mu(\dd y)}{|x-y|^d e^{\rho(\log \frac 1{|x-y|})}}$$
The local finiteness of $\nu_R$ comes from \eqref{clazomp}. \qed

\bibliographystyle{plain}
\bibliography{bibliography.bib}

\end{document}